\documentclass[11pt]{amsart}
\usepackage{amssymb,amsmath,amsthm}
\usepackage{amsfonts}
\usepackage{enumerate}
\usepackage{dsfont}
\usepackage{cite}
\usepackage[colorlinks, citecolor=red]{hyperref}
\usepackage{mathrsfs}
\usepackage{epsfig}
\usepackage{lscape}
\usepackage{subfigure}
\usepackage{epstopdf}
\usepackage{caption}
\usepackage{algorithm}
\usepackage{algpseudocode}
\usepackage{multirow}
\usepackage{geometry}
\usepackage{paralist}
\usepackage{enumerate}
\usepackage{url}
\usepackage{graphicx,cite}
\usepackage{longtable}
\usepackage{tikz}

\newtheorem{definition}{Definition}[section]
\newtheorem{corollary}[definition]{Corollary}
\newtheorem{proposition}[definition]{Proposition}
\newtheorem{theorem}[definition]{Theorem}
\newtheorem{lemma}[definition]{Lemma}

\newtheorem{remark}[definition]{Remark}

\newtheorem{assumption}{Assumption}[section]
\date{}

\begin{document}
\baselineskip 18pt
\title[Randomized iterative methods for GAVE ]{Randomized iterative methods for generalized absolute value equations: Solvability and error bounds}
	
	\author{Jiaxin Xie}
	\address{LMIB of the Ministry of Education, School of Mathematical Sciences, Beihang University, Beijing, 100191, China. }
	\email{xiejx@buaa.edu.cn}

   \author{Hou-Duo Qi}
   \address{Department of Data Science and Artificial Intelligence and Department of Applied Mathematics, The Hong Kong Polytechnic University, Hung Hom, Kowloon, Hong Kong }
   \email{houduo.qi@polyu.edu.hk}

	\author{Deren Han}
	\address{LMIB of the Ministry of Education, School of Mathematical Sciences, Beihang University, Beijing, 100191, China. }
	\email{handr@buaa.edu.cn}

	\maketitle
	
	\begin{abstract}
	Randomized iterative methods, such as the Kaczmarz method and its variants, have gained growing attention due to their simplicity and efficiency in solving large-scale linear systems. Meanwhile, absolute value equations (AVE) have attracted increasing interest due to their connection with the linear complementarity problem. In this paper, we investigate the application of randomized iterative methods to generalized AVE (GAVE). Our approach differs from most existing works in that we tackle GAVE with non-square coefficient matrices. We establish more  comprehensive sufficient and necessary conditions for characterizing  the solvability of GAVE and propose precise error bound conditions. Furthermore, we introduce a flexible and efficient randomized iterative algorithmic framework for solving GAVE, which employs randomized sketching 
	matrices drawn from user-specified distributions. This framework is capable of encompassing many well-known methods,  including  the Picard iteration method and the randomized Kaczmarz method. Leveraging our findings on solvability and error bounds, we establish both almost sure convergence and linear convergence rates for this versatile algorithmic framework. Finally, we present numerical examples to illustrate the advantages of the new algorithms.
	\end{abstract}
	\let\thefootnote\relax

\let\thefootnote\relax

\section{Introduction}

Absolute value equations (AVE) have recently become an active research area due to their relevance to many different fields, including the linear complementarity problem (LCP) \cite{rohn2004theorem}, biometrics \cite{dang2022avet}, game theory \cite{zhou2023semidefinite}, etc.  We consider the following generalized AVE (GAVE)
\begin{equation}\label{GAVE}
	Ax-B|x|=b,
\end{equation}
where  $ A,B\in\mathbb{R}^{m\times n}, b\in\mathbb{R}^m$,  $|x|=(|x_1|,\ldots,|x_n|)^\top$, and  $\top$ denotes the transpose.  
Specifically, when $m=n$ and $B$ is the identity matrix, the GAVE \eqref{GAVE} reduces to the standard AVE \cite{rohn2004theorem}, which is closely related to the LCP, as discussed in detail in Appendix  \ref{sec:appd2}. When $B=0$, the GAVE \eqref{GAVE} becomes the system of linear equations. Over the past two decades, numerous works have examined the AVE problem from various perspectives, such as solvability \cite{wu2020note,wu2021unique,mangasarian2006absolute,hladik2023properties}, error bounds \cite{zamani2023error}, and the development of iterative methods \cite{mangasarian2009generalized,alcantara2023method,chen2023optimal,chen2021inverse,chen2023exact,rohn2014iterative,mangasarian2007absolute}. While most of these works primarily focus on cases where $m=n$, this paper, however, addresses the problem of non-square matrices.

In particular, we present a generic randomized iterative algorithmic framework for solving the  GAVE \eqref{GAVE}.  Starting from any initial point $x^0$, it iterates with the format
\begin{equation}\label{AF-GAVE} 
	x^{k+1}=x^k-\alpha  \frac{A^\top S_kS_k^\top (Ax^k-B|x^k|-b)}{\|S_k^\top A\|^2_2},
\end{equation}
where $\alpha>0$ is the stepsize, $S_k\in\mathbb{R}^{m\times q}$ is a randomized sketching matrix drawn from a user-defined probability space $(\Omega, \mathcal{F}, \mathbf{P})$, and $\|\cdot\|_2$ denotes the Euclidean norm. 
A typical case of the iteration scheme \eqref{AF-GAVE} is the randomized Kaczmarz (RK) method \cite{Str09}. 
For any $ i\in\{1,2,\cdots,m\} $, let $ e_i $ denote the $ i $-th unit vector, $ A_{i, :} $ denote the $ i $-th row of $ A $, and $ b_i $ denote the $ i $-th entry of $ b $. 
Suppose that the coefficient matrix $B=0$, the stepsize $\alpha=1$, and the sampling space $ \Omega=\left\{e_i\right\}_{i=1}^m$ with $ e_i $ being sampled with probability $  \frac{\|A_{i, :}\|^2_2}{\|A\|^2_F}$. Then the iteration scheme \eqref{AF-GAVE} reduces to 
\[
x^{k+1}=x^k-\frac{A_{i_k, :}x^k-b_{i_k}}{\|A_{i_k, :}\|^2_2 }A_{i_k, :}^\top,
\]	
which is exactly the RK method. See  Section \ref{sect-examples} for further discussion on the choice of the probability space $(\Omega,\mathcal{F},\mathbf{P})$ for the recovery of existing methods and the development of new ones.

\subsection{Our contribution}
The main contributions of this work are as follows.
\begin{itemize}
	\item[1.] 
	We present a necessary and sufficient condition for the unique solvability of the GAVE \eqref{GAVE}; see Theorem \ref{thm-solvability}. To the best of our knowledge, this is the first result to characterize the unique solvability of GAVE with non-square coefficient matrices.  To improve the computational tractability of the proposed necessary and sufficient conditions,  we further explore sufficient conditions that ensure the existence and uniqueness of the solution of GAVE. Different from the results in the literature, our conditions incorporate a nonsingular matrix, which makes our criteria more comprehensive; see Theorem \ref{THM-suf}. In addition, we formulate a convex optimization problem to determine the existence of such nonsingular matrix.
	\item[2.] We establish error bounds for the GAVE \eqref{GAVE}, inspired by the recent work of Zamani and Hlad\'{\i}k \cite{zamani2023error} on the standard AVE. Our results extend beyond the scope of their work. Not only do we propose error bounds for GAVE with non-square coefficient matrices, but our error bounds also incorporate a nonsingular matrix; see Theorem \ref{EB-full-rank}. It  greatly generalizes and enhances the results in \cite{zamani2023error}.
	\item[3.] We introduce a simple and versatile randomized iterative algorithmic framework for solving GAVE. Our general framework encompasses several known methods as special cases, including the Picard iteration method, the RK method, and its variants. This enables us to establish new connections between these methods. Additionally, the flexibility of our approach allows us to develop entirely new methods by adjusting the probability space $(\Omega,\mathcal{F},\mathbf{P})$. By leveraging the solvability and error bounds, we demonstrate both almost sure convergence and linear convergence for this general algorithmic framework. Finally, numerical examples illustrate the benefits of the new algorithms.
\end{itemize}


\subsection{Related work}

\subsubsection{Solvability}

Rohn \cite{rohn2004theorem} first introduced the alternative theorem for the unique solvability of the GAVE. This seminal work sparked a large amount  of subsequent research exploring the unique solvability conditions of the GAVE \cite{mangasarian2007absolute,mangasarian2006absolute,rohn2009unique,rohn2014iterative,wu2021unique,wu2018unique,hladik2023some}. One may refer to \cite{kumar2024characterization, Hladik2024an} for recent surveys on the solvability of the GAVE. The unique solvability of GAVE is characterized by certain conditions derived from the analysis of the interval matrix, singular values, spectral radius, and norms of the matrices $A$ and $B$ involved in the equation.  However, these conditions apply only when $A$ and $B$ are square matrices. In this paper, we delve into the unique solvability of the GAVE with non-square coefficient matrices. We establish more comprehensive conditions for characterizing the solvability of the GAVE, expanding the scope of previous research and providing a more inclusive understanding of the problem.



\subsubsection{Error bound} Error bounds play a crucial role in theoretical and numerical analysis of linear
algebra and optimization \cite{pang1997error,zhang2020new,necoara2019linear,higham2002accuracy}. They provide  a measure of the accuracy of an approximation,  helping elucidate the stability and reliability of numerical methods and algorithms. 
In \cite{wang2017numerical,wang2013verification}, the authors established the numerical validation for solutions of standard AVE ($m=n$ and $B=I$) based on the interval methods. Hlad\'{\i}k \cite{hladik2018bounds}  devised outer approximation techniques  and derived an array of bounds for the solution set of the GAVE with square coefficient matrices. 
More recently, Zamani and 
Hlad\'{\i}k \cite{zamani2023error} studied error bounds for standard AVE under the assumption that uniqueness of the solution of AVE is guaranteed. They proposed an \emph{error bound} condition that estimates the distance to the solution $x^*$ from any vector $x\in\mathbb{R}^n$ by the norm of the residual: there exists a constant $\kappa>0$ such that $\|x-x^*\|\leq \kappa\|Ax-|x|-b\|$ for all $x\in\mathbb{R}^n$, where $\|\cdot\|$ represents any vector norm.
In this paper, we extend this line of research by proposing error bounds for GAVE with non-square coefficient matrices. Furthermore, our error bounds incorporate a nonsingular matrix, which significantly broadens and enriches the results presented in \cite{zamani2023error}.

\subsubsection{Existing methods for GAVE}
In recent years, numerous algorithms have been developed to solve the GAVE with square coefficient matrices. Generally, these algorithms fall into four categories: Newton methods \cite{chen2021non,mangasarian2009generalized,caccetta2011globally,bello2016global,haghani2015generalized,hu2011generalized}, Picard iteration methods \cite{rohn2014iterative,salkuyeh2014picard,chen2023exact}, splitting iteration method \cite{ke2017sor,edalatpour2017generalization,chen2023exact}, and concave minimization approach \cite{mangasarian2007absolutec,abdallah2018solving}. For some further comments, please refer to \cite{alcantara2023method, Hladik2024an}. Nevertheless, to our knowledge, there are merely two methods that can handle the GAVE with non-square coefficient matrices. The first is the successive linearization algorithm via concave minimization proposed in \cite{mangasarian2007absolute}. The second is the method of alternating projections (MAP) proposed in \cite{alcantara2023method}. 
See Sections \ref{sect-examples} and \ref{sect:6} for more details and discussion of these methods. Our numerical results illustrate that, in comparison to these existing methods, our approach significantly outperforms in resolving non-square GAVE problems.

\subsubsection{Stochastic algorithms}
Stochastic algorithms, such as the RK method or  the stochastic gradient descent (SGD) method \cite{robbins1951stochastic},  have gained popularity recently due to their small memory footprint and good theoretical guarantees \cite{needell2014stochasticMP,Str09,robbins1951stochastic}.
The Kaczmarz method \cite{Kac37}, also known as the algebraic reconstruction technique (ART) \cite{herman1993algebraic,gordon1970algebraic}, is a classical iterative algorithm used to solve the large-scale linear system of equations. The method alternates between choosing a row of the system and updating the solution based on the projection onto the hyperplane defined by that row.
In the seminal paper \cite{Str09}, Strohmer and Vershynin studied the RK method and proved that if the linear system is consistent, then RK converges linearly in expectation. Since then, a large amount of work has been developed on Kaczmarz-type methods, including accelerated RK methods \cite{liu2016accelerated,han2022pseudoinverse,loizou2020momentum}, randomized block Kaczmarz methods \cite{necoara2019faster,needell2013two,needell2014paved,moorman2021randomized,Gow15}, greedy RK methods \cite{Bai18Gre,Gow19}, randomized sparse Kaczmarz methods \cite{schopfer2019linear,chen2021regularized}, etc. 

In fact, the RK method can be seen as a variant of the SGD method \cite{Han2022-xh,needell2014stochasticMP,ma2017stochastic,zeng2022randomized,jia2024stochastic}.
The SGD method aims to minimize a separable objective function $f(x)=\frac{1}{m}\sum_{i=1}^mf_i(x)$ by stochastically accessing selected components of the objective and taking a gradient step for that component. That is, SGD employs the update rule 
$$x^{k+1}=x^k-\alpha_k\nabla f_{i_k}(x^k),$$
where $\alpha_k$ is the step-size and $i_k$ is selected randomly. This approach allows SGD to make progress towards the minimum of the function using only a subset of the gradient information at each step, which can be computationally advantageous, especially for large-scale problems. When the objective function $f(x)=\frac{1}{2m}\|Ax-b\|^2_2=\frac{1}{2m}\sum_{i=1}^{m}(A_{i,:}x-b_i)^2$, then SGD reduces to the RK method.
Furthermore, the randomized iterative method proposed in this paper can also be  seen as an application of SGD. By incorporating information from the previous iteration, we formulate a stochastic optimization problem and then solve it by using a single-step of SGD; see Section \ref{sect:4} for further discussion. 

Finally, we note that the \emph{iterative sketching} paradigm \cite{Gow15,pilanci2016iterative,derezinski2024recent,dobriban2019asymptotics} is one of the most exciting recent developments in randomized numerical linear algebra for solving large-scale problems. The core idea is to compress a large matrix or dataset into a smaller representation and iteratively refine the solution by solving a sequence of sketched subproblems, rather than addressing the original problem directly. One may refer to \cite{derezinski2024recent} for a detailed description of iterative sketching. In fact, our algorithmic framework \eqref{AF-GAVE} can also be viewed as an implementation of the iterative sketching.



\subsection{Organization}

The remainder of the paper is organized as follows. After introducing some notations and preliminaries in Section $2$, we study the solvability and error bounds in Section $3$. In Section 4, we present the randomized iterative algorithmic framework for solving GAVE and establish its convergence. In Section 5, we mention that by selecting the probability space $(\Omega,\mathcal{F},\mathbf{P})$, the algorithmic framework can recover several existing methods as well as obtain new methods.  In Section 6, 
we perform some numerical experiments to show the effectiveness of the proposed algorithms. We conclude the paper in Section 7. Proofs of the theorems involving lengthy derivations, as well as a discussion of the relationship between the AVE and the LCP, are provided in the appendices.

\section{Basic definitions and Preliminaries}

\subsection{Basic definitions}
For any random variables $\xi_1$ and $\xi_2$, we use $\mathbb{E}[\xi_1]$ and $\mathbb{E}[\xi_1\lvert \xi_2]$ to denote the expectation of $\xi_1$ and the conditional expectation of $\xi_1$ given $\xi_2$. For vector $x\in\mathbb{R}^n$, we use $x_i,x^\top$, and $\|x\|_2$ to denote the $i$-th entry, the transpose, and the Euclidean norm of $x$, respectively. For $d\in\mathbb{R}^n$, $\text{diag}(d)$ stands for the diagonal matrix whose entries on the diagonal are the components of $d$.  

Let  $A$ and $B$ be $m\times n$ matrices.
We use $A_{i,j}$, $A_{i,:}$, $A_{:,j}$, $A^\top$, $A^\dag$, $\|A\|_F$, and $\text{Range}(A)$ to denote the $(i,j)$-th component, the $i$-th row, the $j$-th column, the transpose, the Moore-Penrose pseudoinverse, the Frobenius norm, and the column space of $A$, respectively. If $A$ is nonsingular, then $A^\dagger=A^{-1}$. The singular values of $A$ are $\sigma_1(A)\geq\sigma_2(A)\geq\ldots\geq\sigma_{\min\{m,n\}}(A)\geq0$. Given $\mathcal{J} \subseteq [m]:=\{1,\ldots,m\}$, the cardinality of the set $\mathcal{J} $ is denoted by $\operatorname{card}(\mathcal{J})$ and the complementary set of $\mathcal{J} $ is denoted by $\mathcal{J} ^c$, i.e. $\mathcal{J} ^c=[m]\setminus \mathcal{J}$. We use $A_{\mathcal{J} ,:}$ and $A_{:,\mathcal{J} }$ to denote the row and column submatrix indexed by $\mathcal{J} $, respectively.  The matrix inequality $A \leq B$ is  understood entrywise and the interval matrix $[A,B]$ is defined as $[A,B]:= \{C \mid A \leq C \leq B\}$. We use $I_n\in\mathbb{R}^{n\times n}$ to denote the identity matrix.  

A symmetric matrix $M\in\mathbb{R}^{n\times n}$ is said to be positive semidefinite if $x^\top M x\geq 0$ holds for any
$x\in\mathbb{R}^n$, and  $M$ is positive definite if $x^\top M x> 0$ holds for any nonzero $x\in\mathbb{R}^n$. We use $\lambda_1(M)\geq \lambda_2(M)\geq\ldots\geq\lambda_{n}(M)$ to denote the eigenvalues of $M$. We can see that $\|A\|_2=\sigma_{1}(A)$, $\|A\|_F=\sqrt{\sum\limits_{i=1}^n \sigma^2_i(A)}$, and $\sigma_i(A)=\sqrt{\lambda_i(A^\top A)}$ for any $i\in[n]$.
Letting $\Lambda=\text{diag}(\lambda_1(M),\ldots,\lambda_n(M))$ and $M=U\Lambda U^\top$ denote the eigenvalue decomposition of $M$, we denote  $M^{\frac{1}{2}}=U\Lambda^{\frac{1}{2}} U^\top$ with $\Lambda^{\frac{1}{2}}=\text{diag}(\sqrt{\lambda_1(M)},\ldots,\sqrt{\lambda_n(M)})$. For any
two matrices $M$ and $N$, we write $M\succeq N$ ($M\succ N $) to represent $M-N$ is positive semidefinite (definite). For any $M\succeq0$, we define $\|x\|_{M}:=\sqrt{\langle x,Mx\rangle}=\|M^{\frac{1}{2}}x\|_2$.

We use $\mathcal{X}^*$ to denote the solution set of the GAVE \eqref{GAVE}. The generalized Jacobian matrices \cite{clarke1990optimization} are used in the presence of nonsmooth
functions. Let $F: \mathbb{R}^n\to\mathbb{R}^m$ be a locally Lipschitz function. The generalized
Jacobian of $F$ at $\hat{x}$, denoted by $\partial F(\hat{x})$, is defined as
$$
\partial F(\hat{x}):=\operatorname{co}\left(\left\{\lim_{n\to\infty}\nabla F(x_n) : x_n\to \hat{x},x_n\notin \mathcal{X}_f\right\}\right),
$$
where $\mathcal{X}_f$ is the set of points at which $f$ is not differentiable and $\operatorname{co}(\mathcal{S})$ denotes the
convex hull of a set $\mathcal{S}$.  We use $\operatorname{dist}_{\mathcal{S}}(x)$ to denote the distance from $x$ to the set $\mathcal{S}$, i.e., $\operatorname{dist}_{\mathcal{S}}(x):=\inf_{y\in\mathcal{S}}\|x-y\|_2$.

\subsection{Some useful lemmas}
In this subsection, we recall some known results that we will need later on. 
By the definition of singular value, we know that for any $P,Q\in\mathbb{R}^{m\times n}$, $\sigma^2_{i}(P^\top Q)=\lambda_i(Q^\top PP^\top Q)=\lambda_i(Q^\top(PP^\top)^{\frac{1}{2}} (PP^\top)^{\frac{1}{2}} Q)=\sigma^2_{i}((PP^\top)^{\frac{1}{2}} Q)$. Therefore, we can conclude the following lemma.
\begin{lemma}\label{singular-value-equ}
	Let $P,Q\in\mathbb{R}^{m\times n}$ and $\ell=\min\{m,n\}$. Then for any $i\in[\ell]$, $\sigma_{i}(P^\top Q)=\sigma_{i}((PP^\top)^{\frac{1}{2}} Q)$.
\end{lemma}

\begin{lemma}[\cite{zhang2011matrix}]
	\label{singular-value-inequ}
	Let $P,Q\in\mathbb{R}^{m\times n}$. Then 
	$\sigma_{i}(P+Q)\geq\sigma_{i}(P)-\|Q\|_2$.
\end{lemma}

In order to proceed, we shall propose a basic assumption on the probability space $(\Omega, \mathcal{F}, \mathbf{P})$ used in this paper.



\begin{assumption}
	\label{Ass}
	Let $(\Omega, \mathcal{F}, \mathbf{P})$  be  the probability space from which the sampling matrices are drawn. We assume that  $\mathop{\mathbb{E}}_{S\in\Omega} \left[S S^\top\right]$ is a positive definite matrix.
\end{assumption}

The following lemma is crucial for our convergence analysis.

\begin{lemma}[Lemma 2.3, \cite{lorenz2023minimal}]
	\label{positive}
	Let $S\in\mathbb{R}^{m\times q}$ be a real-valued random variable defined on a probability space $(\Omega,\mathcal{F}, \mathbf{P})$. Suppose that
	$
	\mathbb{E}\left[SS^\top\right]
	$
	is a positive definite matrix and $A\in\mathbb{R}^{m\times n}$ with $A\neq 0$.  Then
	$$
	H:=\mathbb{E}\left[\frac{SS^\top}{\|S^\top A\|^2_2}\right]
	$$
	is well-defined and positive definite, here we define $\frac{0}{0}=0$.
\end{lemma}

\section{Solvability and error bounds} \label{Section-Solvability}

This section is organized as follows. We first study a few of necessary and sufficient conditions
for the solvability of GAVE \eqref{GAVE} (Section~\ref{Subsection-Solvability-NS}).
We then propose a set of sufficient conditions, which can be checked via convex optimization (Section~\ref{section-3-1-1}). Under the solvability, we investigate how far a given point is from the
solution set in terms of residual errors. Such results are collectively put in the error bound Section~\ref{Subsection-Error-Bounds}.

\subsection{Solvability: necessary and sufficient conditions} \label{Subsection-Solvability-NS}

For the case where $m = n$, Theorem \ref{unique-solution} proposed by Wu and Shen \cite{wu2021unique} offers a characterization of the unique solvability of the GAVE \eqref{GAVE}.


\begin{theorem}[Theorem~3.2, \cite{wu2021unique}] \label{unique-solution}
	Suppose that $m= n$. The GAVE \eqref{GAVE} has a unique solution for any  $b\in\mathbb{R}^n$
	if and only if for any $D\in[-I_n,I_n]$, the matrix $A + BD$ is nonsingular.
\end{theorem}
For arbitrary values of $m$ and $n$, the unique solvability of the GAVE \eqref{GAVE} can be characterized by the following result, which can be viewed as a generalization of Theorem \ref{unique-solution}. To the best of our knowledge, this is the first result to characterize the unique solvability of GAVE with non-square coefficient matrices. 
\begin{theorem}\label{thm-solvability} 
	The GAVE \eqref{GAVE} has a unique solution for any  $b\in\mathbb{R}^m$
	if and only if $m=n$ and for any $D\in[-I_n,I_n]$, the matrix $A + BD$ is nonsingular.
\end{theorem}

Next, we will present a crucial lemma that will not only help us prove Theorem \ref{thm-solvability}, but also its proof will provide valuable insights into the solvability of the GAVE.
For any $A,B\in\mathbb{R}^{m\times n}$, we define the map $F_{A,B}:\mathbb{R}^n\to\mathbb{R}^m$ as
\begin{equation}\label{mapF}
	F_{A,B}(x):=Ax-B|x|.
\end{equation}
We have the following conclusion for the map $F_{A,B}$.

\begin{lemma} \label{Lemma-surjective-no}
	Suppose that $m\neq n$. 
	Then for any $A$ and $B$, the map $F_{A,B}$ defined in \eqref{mapF} is not a bijection map. 
	In particular, the following statements hold.
	
	\begin{itemize}
		\item[(i)] 	
		Suppose $m<n$.	 For any $A,B\in\mathbb{R}^{m\times n}$, there exists $b\in\mathbb{R}^m$ such that the  GAVE \eqref{GAVE}
		admits infinitely many solutions.
		
		\item[(ii)] Suppose $m>n$.	 For any $A,B\in\mathbb{R}^{m\times n}$, there exists $b\in\mathbb{R}^m$ such that the  GAVE \eqref{GAVE} has no solution.
		
	\end{itemize}

\end{lemma}

\begin{proof}
	For the case where $m<n$, we will prove that for any $A$ and $B$, the map $F_{A,B}$  is not injective.
	Define $c_1:=A{\bf e}-B{\bf e}$, where ${\bf e}=(1,\ldots,1)^\top\in\mathbb{R}^n$. So in the first orthant, there exists a vector 
	$x^* ={\bf e}$ which  satisfies $F_{A,B}(x^*)=c_1$. Since 
	$$
	\text{dim}\left(\text{Null}(A -B)\right)\geq n-m>0,
	$$ we know that the null space of $A-B$ is non-empty. 
	Note that $x^*={\bf e}$ lies in the interior of the first orthant, hence there exist infinitely many vectors $\tilde{x}$ such that $F_{A,B}(\tilde{x})=c_1$. This implies that the map $F_{A,B}$  is not injective.

	For the case where $m>n$, we will prove that for any $A$ and $B$, the map $F_{A,B}$ is not  surjective. For any $s\in\{-1,1\}^n$, we define
	$$
	E_{s}:=\{ Ax-B|x| \mid x\in\mathbb{R}^n,D_s x\geq0\},
	$$
	where $D_s=\operatorname{diag}(s)$.
	Since
	$$
	\text{dim}(E_s)\leq n<m,
	$$
	we have that the Lebesgue measure of $E_s$ is equal to zero. Thus,  $$\text{Image}(F_{A,B})=\mathop{\cup}\limits_{s\in\{-1,1\}^n} E_s \neq \mathbb{R}^m,$$ which implies that the map $F_{A,B}$ is not  surjective. The statements in (i) and (ii) are straightforward consequences of
	$F_{A, B}$ not being a bijection map.
\end{proof}

Now, we are ready to prove Theorem \ref{thm-solvability}.
\begin{proof}[Proof of Theorem \ref{thm-solvability}]
	``If''. Obvious in view of Theorem \ref{unique-solution}.
	``Only if''. If $m\neq n$, Lemma \ref{Lemma-surjective-no} indicates that there cannot exist matrices $A \in \mathbb{R}^{m\times n}$ and $B \in \mathbb{R}^{m\times n}$ such that the GAVE \eqref{GAVE} has a unique solution for any $c \in \mathbb{R}^m$. This implies that $m$ must be equal to $n$. Combining this result with Theorem \ref{unique-solution} leads to the conclusion of this theorem.
\end{proof}

Lemma~\ref{Lemma-surjective-no} provides a rather negative assessment for the case $m \not= n$:  
There does not exist $A\in\mathbb{R}^{m\times n}$ and $B\in\mathbb{R}^{m\times n}$ such that for any $b\in\mathbb{R}^m$, the  GAVE  \eqref{GAVE}	has  unique solution.
Furthermore, the situation is different for the case $m <n$ and $m >n$.
If we are to characterize its solvability in terms of a necessary and sufficient condition, the vector $b$ must 
be related to the choice of the matrices $A$ and $B$. 
The next result states such a condition.

\begin{theorem}\label{sn-condition}
	The GAVE \eqref{GAVE} is solvable if and only if there is $s\in\{-1,1\}^n$ such that
	$$
	\{x\mid \left(A-BD_s\right)x=b\}\cap\{x\mid D_sx\geq0\}\neq \emptyset,
	$$
	where $D_s=\operatorname{diag}(s)$.
\end{theorem}

\begin{proof} The ``If'' part is obvious, so we focus on the ``Only if'' part.
	Since GAVE \eqref{GAVE} is solvable, we know that there exist $x^*$ and $s^*$ such that $Ax^*-BD_{s^*}x^*=b$ and $D_{s^*}x^*\geq 0$. Hence, $
	\emptyset \neq \{x\mid \left(A-BD_{s^*}\right)x=b\}\cap\{x\mid D_{s^*}x\geq0\} \supseteq\{x^*\}.
	$
\end{proof}

If $m=n$ and $B=I$, Theorem \ref{sn-condition} recovers Proposition $2.1$  in \cite{hladik2023properties},  which provides a combinatorial structure for the solution set of the  standard AVE.
If $B=0$, Theorem \ref{sn-condition} reduces to the fundamental result in linear algebra: the linear system $Ax=b$ is solvable (consistent) if and only if $b\in\operatorname{Range}(A)$.
However, checking the condition from Theorem \ref{sn-condition} might not be easy. 
Therefore, a more efficiently computable condition is of interest. 
Such conditions are often only sufficient, as investigated in the next subsection.


\subsection{Solvability: sufficient conditions} 
\label{section-3-1-1}

We provide one set of sufficient conditions, whose validity can be verified via a convex optimization.
The sufficient conditions are stated in the following result. We note that the proof of this result, as well as the proofs of other theorems with lengthy derivations, are all moved to Appendix \ref{sec:appd}.

\begin{theorem}\label{THM-suf}
	Let $\ell := \min\{m, n\}$.
	Suppose there exists a nonsingular matrix $M\in\mathbb{R}^{m\times m}$ such that $\sigma_{\ell}(MA)> \|MB\|_2$.
	The following results hold.
	\begin{itemize} 
		\item[(i)] For $m < n$, it holds that the GAVE \eqref{GAVE} is solvable for any $b\in\mathbb{R}^m$.
		
		\item[(ii)] For $m=n$, it holds that the GAVE \eqref{GAVE}  has a unique solution for any $b\in\mathbb{R}^m$.
		
		\item[(iii)] For the case $m >n$, if we further assume that the solution set $\mathcal{X}^*$ is non-empty, 
		then $\mathcal{X}^*$ is singleton.
		
	\end{itemize} 
\end{theorem}



We make following important points. 

\begin{remark} \label{Remark-Sufficient-Conditions-1}
	{\bf (R1)} It is evident that the condition $\sigma_{\ell}(MA) > \|MB\|_2$ for some nonsingular matrix $M$ is weaker than 
	the condition $\sigma_{\ell}(A) > \|B\|_2$
	by simply setting $M = I_m$. The following examples show that it is strictly weaker
	for all three cases.
	For those examples, we have $\sigma_{\ell}(MA) > \|MB\|_2$, but $\sigma_{\ell}(A) < \|B\|_2$.
	$$\begin{aligned}
		& (\mbox{Case} \ m <n) \qquad
		A=\begin{bmatrix} 2 & 2 & 6 \\ -3 & -6 &8 \end{bmatrix},\
		B=\begin{bmatrix} 4 & 1 &0 \\ 3 & -1&-4 \end{bmatrix}, \
		\text{and} \ M=\begin{bmatrix} 8 & -1 \\ -1 & 8 \end{bmatrix},\\
		& \quad \qquad \qquad \qquad 	\sigma_{2}(A)\approx 5.6807<\|B\|_2\approx 5.7780 \ \mbox{and} \ \sigma_{2}(MA)\approx 48.1674>\|MB\|_2\approx 41.0904.\\
		&(\mbox{Case} \ m =n) \qquad 
		A=\begin{bmatrix} -7 & 11 \\ 10 & -2 \end{bmatrix},\
		B=\begin{bmatrix} -2 & 2 \\ 6 & 0 \end{bmatrix},\
		\text{and} \ M=\begin{bmatrix} 13 & 2 \\ 2 & 11 \end{bmatrix}, \\
		& \quad \qquad \qquad \qquad
		\sigma_{2}(A)\approx 6.2658<\|B\|_2\approx 6.3592 \ \mbox{and} \ \sigma_{2}(MA)\approx81.2427>\|MB\|_2\approx 63.592.\\
		& (\mbox{Case} \ m > n) \qquad  
		A=\begin{bmatrix} -6 & -9  \\ 6 & -4 \\5&-2 \end{bmatrix},\  B=\begin{bmatrix} -2 & 6\\2 &-4 \\ -6 & -5 \end{bmatrix},\  \text{and} \ M=\begin{bmatrix} 45 &13 & 0 \\ 13 & 33&24\\0&24&24 \end{bmatrix}, \\
		& \quad \qquad \qquad \qquad
		\sigma_{2}(A)\approx 8.8826<\|B\|_2\approx 8.9327 \ \mbox{and}  \ \sigma_{2}(MA)\approx 401.4896>\|MB\|_2\approx 360.9529 .	
	\end{aligned}$$
	
	{\bf (R2)}
	For the special case where $m=n$, the condition  $\sigma_{n}(A)> \|B\|_2$ appears in \cite[Theorem 2.1]{wu2020note}.
	Our condition is weaker as shown in (R1).
	
	{\bf (R3)} For the case \(m > n\), we have to assume that the solution set \(\mathcal{X}^*\) is nonempty, as Lemma \ref{Lemma-surjective-no} (ii) implies that the GAVE \eqref{GAVE} is typically unsolvable.
	
	{\bf (R4)} If the condition $\sigma_{m}(MA) > \|MB\|_2$ is satisfied for one matrix $M$, then there are infinitely many such matrices $\{ \tau M\}_{\tau \not=0}$. This is because for any $\tau \not=0$, we have
	$
	\sigma_{m}(\tau MA) = |\tau| \sigma_{m}( MA) > |\tau| \|MB\|_2 = \| \tau M B \|_2.
	$	
	
	{\bf (R5)} Our final comment is on the rank conditions of $A$ and $B$ for the solution uniqueness.
	Both Theorem~\ref{thm-solvability} (for the case $m=n$) and Theorem~\ref{THM-suf}(iii) (for the case
	$m>n)$ require the matrix $A$ to be of full column rank (necessary condition). 
	Neither of them requires the same property of $B$. 
	It is worth noting that the full column rank property for both $A$ and $B$ is not sufficient for 
	solution uniqueness.	
\end{remark}

\subsubsection{The existence of the matrix $M$}
\label{section-3-12}
In this subsection, we demonstrate that the problem of finding a nonsingular matrix $M$ such that $\sigma_{\ell}(MA) > \|MB\|_2$ 
can be reformulated as a convex optimization problem, where $\ell=\min\{m,n\}$. For convenience,  we assume that $m\geq n$ in the following discussion. 

From Lemma \ref{singular-value-equ}, we know that  if  there exists a nonsingular matrix $M$  such that $\sigma_{n}(MA) > \|MB\|_2$, then it holds that  $M_1:=(M^\top M)^\frac{1}{2}$ satisfies $\sigma_{n}(M_1A) > \|M_1B\|_2$. Additionally,  from Remark \ref{Remark-Sufficient-Conditions-1} (R4), we know that for any $\zeta\neq 0$, $M_2:=\zeta M$ also satisfies $\sigma_{n}(M_2A) > \|M_2B\|_2$. 
Therefore, we only need to consider the constraint
$$\mathcal{M}:=\{M :  M \succeq I\}$$
to find a nonsingular matrix $M$ that satisfies $\sigma_{n}(M^\frac{1}{2}A) > \|M^\frac{1}{2}B\|_2$.
Besides, it can be verified that $\mathcal{M}$ is a convex set.

For any given $A,B\in\mathbb{R}^{m\times n}$, we define $\Phi_{A,B}:\mathbb{R}^{m\times m}\to\mathbb{R}$ as
$$\Phi_{A,B}(M):=\lambda_1(B^\top MB)-\lambda_{n}(A^\top MA).$$ 
Since for any fixed $B$, $\lambda_1(B^\top MB)$ is a convex function, and for any fixed $A$, $\lambda_{n}(A^\top MA)$ is a concave function; see \cite[Section 3.2.3]{boyd2004convex}. Hence, we have that $\Phi_{A,B}$ is a convex function. 
Therefore, we can solve the convex optimization problem
\begin{equation}
	\label{find-M}
	\min \  \Phi_{A,B}(M) \ \
	\text{subject to} \ M \in\mathcal{M}
\end{equation}
to determine the existence of the nonsingular matrix $M$. If $M^*$ is the optimal solution of \eqref{find-M} and $\Phi_{A,B}(M^*)<0$, then $M^*$ is the desired nonsingular matrix such that $$\sqrt{\lambda_n(A^\top M^*A)}=\sigma_{n}((M^*)^{\frac{1}{2}}A) > \|(M^*)^{\frac{1}{2}}B\|_2=\sqrt{\lambda_1(B^\top M^*B)}.$$
Otherwise, we know that there does not exist a nonsingular matrix that satisfies $\sigma_{n}(MA) > \|MB\|_2$.
We note that \eqref{find-M} is indeed a semidefinite programming (SDP)  \cite{vandenberghe1996semidefinite,wolkowicz2012handbook}, which can be solved efficiently.


\subsection{Error bounds}
\label{Subsection-Error-Bounds}

Inspired by the recent work of Zamani and Hlad\'{\i}k \cite{zamani2023error},  this subsection investigates the error bounds for the solvable GAVE \eqref{GAVE}.
In fact, based on the locally upper Lipschitzian property of polyhedral set-valued mappings as described in Proposition $1$ in \cite{robinson1981some}, the GAVE \eqref{GAVE} exhibits the local error bounds property. Specifically, there exist $\varepsilon>0$ and $\kappa>0$ such that when $\|Ax-B|x|-b\|_2<\varepsilon$, it holds
\begin{equation}\label{localEB}
	\kappa\text{dist}_{\mathcal{X}^*}(x) \leq \|Ax-B|x|-b\|_2.
\end{equation}
However, in general, the global error bounds property does
not hold necessarily, as demonstrated in Example $3$ in \cite{zamani2023error}. Next, we will provide several sufficient conditions under which the global error bounds property holds.

\begin{theorem}\label{EB-general}
	Let $\mathcal{X}^*$ be non-empty. If zero is the unique solution of $Ax-B|x|=0$, then
	there exists $\kappa> 0$ such that
	\begin{equation}\label{EB-L}
		\kappa \operatorname{dist}_{\mathcal{X}^*}(x)\leq \| Ax-B|x|-b\|_2, \ \forall x\in\mathbb{R}^n.
	\end{equation}
\end{theorem}

\begin{proof}
	The idea of the proof is similar to that of Theorem $12$ in \cite{zamani2023error} or Theorem  2.1 in \cite{mangasarian1994new}. Suppose to the
	contrary that \eqref{EB-L} does not hold. Hence, for each $k\in\mathbb{N}$, there exists $x^k$ such that for any fixed $\bar{x}\in \mathcal{X}^*$
	\begin{equation}\label{EB-prf1}
		\|x^k-\bar{x}\|_2\geq \text{dist}_{\mathcal{X}^*}(x^k)>k\| Ax^k-B|x^k|-b\|_2.
	\end{equation}
	Due to the local error bounds property \eqref{localEB}, there exists $\varepsilon$ such that
	$\| Ax^k-B|x^k|-b\|_2>\varepsilon$ for each $k\geq k_0$, where $k_0$ is sufficiently large. Consequently, $\|x^k-\bar{x}\|_2$ tends to infinity as $k\to\infty$. Choosing subsequences if necessary, we may assume that $\frac{x^k}{\|x^k\|_2}$ goes to a non-zero vector $d$. By dividing both sides of \eqref{EB-prf1} by $k\|x^k\|_2$ and taking the limit as $k$ goes to infinity, we get
	$$
	Ad-B|d|=0,
	$$
	which contradicts the assumptions.
\end{proof}

Theorem \ref{THM-suf} (ii) and (iii) have already provided sufficient conditions to guarantee that zero is the unique solution of $Ax-B|x|=0$. In addition to those results, the following proposition provides a simpler sufficient condition to ensure the uniqueness of the solution set of $Ax-B|x|=0$.

\begin{proposition}\label{prop-xie-24-2-3}
	Suppose there exists an index $i\in[m]$ such that  $|A_{i,:}|<B_{i,:}$. Then zero is the unique solution of $Ax-B|x|=0$.
\end{proposition}
\begin{proof}
	Based on the assumption, for any  $j\in[n]$, it holds that $A_{i,j}x_j<B_{i,j}|x_j|$ if $x_j\neq0$. Consequently, if $x\neq 0$, we have $A_{i,:}x-B_{i,:}|x|<0$. Thus,  if $x\neq 0$, we have $Ax-B|x|\neq 0$. Therefore, zero is the unique solution of $Ax-B|x|=0$.
\end{proof}

\begin{theorem}\label{EB-full-rank}
	Suppose that  $\mathcal{X}^*$ is non-empty and for any $D\in[-I_n,I_n]$, the matrix $A+BD$ is full column rank.  Then for any $x^*\in\mathcal{X}^*$ and nonsingular matrix $M\in\mathbb{R}^{m\times m}$,
	$$
	\|x-x^*\|\leq \max\limits_{D\in[-I_n,I_n]}\|\left(MA+MBD\right)^{\dagger}\|\cdot \|M( Ax-B|x|-b)\|, \ \forall x\in\mathbb{R}^n,
	$$
	where $\|\cdot\|$ represents any vector norm and its induced norm.
\end{theorem}

Note that under the assumption of Theorem \ref{EB-full-rank}, it holds that $m\geq n$. Assuming $m=n$ and if the matrix $A + BD$ is full column rank (nonsingular) for every $D\in[-I_n,I_n]$, Theorem \ref{unique-solution} ensures a unique solution to the GAVE \eqref{GAVE}. Therefore, Theorem \ref{EB-full-rank} yields the following corollary.

\begin{corollary}\label{EB-full-1}
	Assume that $m=n$ and $A + BD$ is nonsingular for any $D\in[-I_n,I_n]$. Then for any nonsingular matrix $M\in\mathbb{R}^{n\times n}$,
	$$
	\|x-x^*\|\leq \max\limits_{D\in[-I_n,I_n]}\|\left(MA+MBD\right)^{-1}\|\cdot\| M(Ax-B|x|-b)\|, \ \forall x\in\mathbb{R}^n,
	$$
	where $\|\cdot\|$ represents any vector norm and its induced norm.
\end{corollary}

\begin{remark} \label{Remark-Error-Bounds}
	If $B=I$ and $M=I$, Corollary \ref{EB-full-1} recovers \cite[Thorem~7]{zamani2023error}, which provides an error bound for the standard system of AVE.
	Furthermore, the nonsingular matrix $M$ acts like a preconditioner so that the following two systems are
	equivalent:
	\[
	Ax - B|x| = b \qquad \Longleftrightarrow \qquad
	MAx - MB |x| = Mb .
	\]
	If we further assume that $B$ is nonsingular, we may choose $M=B^{-1}$ so that the system $B^{-1}A x - |x| = B^{-1}b$
	holds. For this type of AVE system,\cite[Theorem~7]{zamani2023error} implies the following
	\begin{align*}
		\| x - x^*\| & \le \max_{D \in [-I_n, I_n]} \| (B^{-1}A + D)^{-1} \| \| B^{-1}A - |x| - B^{-1}b\|\\
		& = \max_{D \in [-I_n, I_n]} \| (MA + MBD)^{-1} \| \| M(A - B |x| - b)\|.
	\end{align*}
	Therefore, under the assumption that $B$ is nonsingular, Corollary~\ref{EB-full-1}
	with the particular choice $M=B^{-1}$ can be derived from
	\cite[Thorem~7]{zamani2023error}. We note that our result holds for any nonsingular $M$ with $B$ not necessarily
	nonsingular.
\end{remark}

Theorem \ref{EB-full-rank} also yields the following corollary.

\begin{corollary}\label{EB-2-norm}
	Let $\mathcal{X}^*$ be non-empty and $M\in\mathbb{R}^{m\times m}$ is a nonsingular matrix. Suppose that $m\geq n$ and $\sigma_{n}(MA)>\|MB\|_2$. Then
	$$
	\| x-x^*\|_2\leq \frac{1}{\sigma_{n}(MA)-\|MB\|_2} \| Ax-B|x|-b\|_{M^\top M}.
	$$
\end{corollary}
\begin{proof}
	For any $D\in[-I_n,I_n]$, we have  
	$$\sigma_{n}(MA+MBD)\geq\sigma_{n}(MA)-\|MBD\|_2\geq \sigma_{n}(MA)-\|MB\|_2>0,$$
	where the first inequality follows from Lemma \ref{singular-value-inequ} and the second inequality follows from $\|D\|_2\leq 1$.
	Therefore, the matrix $MA+MBD$ has full column rank, and thus, $A+BD$ is also full column rank since $M$ is nonsingular.
	Moreover, we have
	$$\max\limits_{D\in[-I_n,I_n]}\|\left(MA+MBD\right)^{\dagger}\|_2=\sigma^{-1}_{n}(MA+MBD)\leq \frac{1}{\sigma_{n}(MA)-\|MB\|_2}.$$
	Hence from Theorem \ref{EB-full-rank}, we have 
	$$
	\begin{aligned}
		\| x-x^*\|_2&\leq\frac{1}{\sigma_{\min}(MA)-\|MB\|_2} \| M(Ax-B|x|-b)\|_2\\
		&= \frac{1}{\sigma_{n}(MA)-\|MB\|_2} \| Ax-B|x|-b\|_{M^\top M}
	\end{aligned}
	$$
	as desired.
\end{proof}

\section{Randomized iterative methods for GAVE}
\label{sect:4}

In this section, we present our randomized iterative method for solving the GAVE \eqref{GAVE} and analyze its convergence properties. Randomization is incorporated into our method through a user-defined probability space $(\Omega,\mathcal{F},\mathbf{P})$ that describes an ensemble of randomized sketching matrices $S\in\mathbb{R}^{m\times q}$. The selection of the probability space should ideally be based on the specific problem at hand, as it can impact the error bounds and convergence rates of the method.  Our approach and underlying theory accommodate a wide range of probability distributions; see Section \ref{sect-examples} for further discussion.

At the $k$-th iteration, we consider the following stochastic optimization problem
\begin{equation}
	\label{sop-1}
	\mathop{\min}\limits_{x \in \mathbb{R}^n} f^k(x):=\mathop{\mathbb{E}}\limits_{S\in\Omega}\left[ f^k_S(x)\right],
\end{equation}
where $f^k_S(x):=\frac{1}{2} \left\|S^\top(Ax-B|x^k|-b)\right\|^2_2$ with $S$ being a random variable in $(\Omega, \mathcal{F}, \mathbf{P})$.  Starting from $x^k$, we employ only one step of the SGD method to solve the stochastic optimization problem \eqref{sop-1}
$$
x^{k+1}=x^k-\alpha_k\nabla f^k_{S_k}(x^k)=x^k-\alpha_k A^\top S_k S_k^\top(Ax^k-B|x^k|-b),
$$
where $S_k$ is drawn from the sample space $\Omega$ and $\alpha_k$ is the step-size. Particularly, we choose $\alpha_k=\alpha/\|S_k^\top A\|^2_2$ with $0<\alpha\leq 1$. 
Now, we are ready to state the randomized iterative method for solving the GAVE \eqref{GAVE} described in Algorithm \ref{RIM}. 

\begin{algorithm}[htpb]
	\caption{Randomized iterative method (RIM) for GAVE}
	\label{RIM}
	\begin{algorithmic}
		\Require $A, B \in \mathbb{R}^{m \times n}$, $b \in \mathbb{R}^m$, probability spaces $(\Omega, \mathcal{F}, \mathbf{P})$, $\alpha\in(0,1]$, $k=0$, and the initial point $x^0 \in \mathbb{R}^n$.
		\begin{enumerate}
			\item[1.] Randomly select a randomized sketching matrix $S_k\in \Omega$.
			\item[2.]  Update 
			\begin{equation}\label{one-step}
				x^{k+1}=x^k-\alpha\frac{A^\top S_kS_k^\top (Ax^k-B|x^k|-b)}{\|S_k^\top A\|^2_2}.
			\end{equation}
			\item[3.] If a stopping rule is satisfied, stop and go to output. Otherwise, set $k=k+1$ and return to Step $1$.
		\end{enumerate}
		
		\Ensure
		The approximate solution $x^k$.
	\end{algorithmic}
\end{algorithm}

\subsection{Convergence analysis}

Let $x^*$ be any element of $\mathcal{X}^*$. For Algorithm \ref{RIM}, we establish the following convergence result. The parameters $\alpha$ and $\xi$ are chosen to ensure that $0 < \alpha \leq 1$, which is a necessary condition for convergence.  

\begin{theorem}
	\label{THM-asym}
	Assume that $\mathcal{X}^*$ is nonempty and the probability spaces $(\Omega, \mathcal{F}, \mathbf{P})$ satisfy Assumption \ref{Ass}. Let $H=\mathbb{E}\left[\frac{SS^\top}{\|S^\top A\|^2_2}\right]$ and  $\{x^k\}_{k \geq 0}$ be the iteration sequence generated by Algorithm \ref{RIM}.
	\begin{itemize}
		\item[(\romannumeral1)] If $\sigma_{n}(H^{\frac{1}{2}}A)=\|H^{\frac{1}{2}}B\|_2$ and $\alpha\in(0,1)$, then at least one subsequence of $\{x^k\}_{k\geq0}$ converges a.s. to a point in the set $\mathcal{X}^*$ and $Ax^k-B|x^k|-b$ converges a.s. to zero.
		\item[(\romannumeral2)] If $\sigma_{n}(H^{\frac{1}{2}}A)>\|H^{\frac{1}{2}}B\|_2$ and $\alpha=(2-\xi)\frac{\sigma_{n}(H^{\frac{1}{2}}A)}{\sigma_{n}(H^{\frac{1}{2}}A)-\|H^{\frac{1}{2}}B\|_2}$ with $\xi\in\left[\frac{\sigma_{n}(H^{\frac{1}{2}}A)+\|H^{\frac{1}{2}}B\|_2}{\sigma_{n}(H^{\frac{1}{2}}A)},2\right)$,
		then $x^*$ is the unique solution and 
		$$	
		\mathbb{E}\left[\|x^k-x^*\|^2_2\right]\leq \left(1-(2-\xi)\xi\sigma_{n}^2(H^{\frac{1}{2}}A)\right)^k \|x^0-x^*\|^2_2.
		$$
		\item[(\romannumeral3)] If zero is the unique solution of $Ax-B|x|=0$, $\sigma_{n}^2(H^{\frac{1}{2}}A)>\|H^{\frac{1}{2}}B\|^2_2-\lambda_{\min}(H)\kappa^2$ with $\kappa$ being given in Theorem \ref{EB-general},  and $\alpha=(2-\xi)\frac{\sigma_{n}^2(H^{\frac{1}{2}}A)-\|H^{\frac{1}{2}}B\|^2_2+\lambda_{\min}(H)\kappa^2}{2\lambda_{\min}(H)\kappa^2}$ with $\xi\in\left[2\frac{\sigma_{n}^2(H^{\frac{1}{2}}A)-\|H^{\frac{1}{2}}B\|^2_2}{\sigma_{n}^2(H^{\frac{1}{2}}A)-\|H^{\frac{1}{2}}B\|^2_2+\lambda_{\min}(H)\kappa^2},2\right)$, then 
		$$	
		\mathbb{E}\left[\operatorname{dist}^2_{\mathcal{X^*}}(x^k)\right]\leq \left(1-(2-\xi)\xi\frac{\sigma_{n}^2(H^{\frac{1}{2}}A)-\|H^{\frac{1}{2}}B\|^2_2+\lambda_{\min}(H)\kappa^2}{4\lambda_{\min}(H)\kappa^2}\right)^k \operatorname{dist}^2_{\mathcal{X^*}}(x^0).
		$$
	\end{itemize}
\end{theorem}


\begin{remark}
	If $B=0$, Theorem \ref{THM-asym} $(\romannumeral1)$ aligns with the almost sure convergence result obtained through the properties of stochastic Quasi-Fej{\'e}r sequences (see Remark 3.2 in \cite{briceno2022random}). Theorem \ref{THM-asym} $(\romannumeral2)$ and $(\romannumeral3)$ can readily derive the almost sure convergence result. In fact, the proof of Theorem \ref{THM-asym} $(\romannumeral2)$ reveals that
	$
	\mathbb{E}[\|x^{k+1}-x^*\|^2_2\mid x^k]\leq \|x^{k}-x^*\|^2_2-(2-\xi)\xi\sigma_{n}^2(H^{\frac{1}{2}}A)\cdot\|x^{k}-x^*\|^2_2
	$ (see \eqref{proof-xie-2025-2-3-1}). Consequently, by applying the supermartingale convergence lemma (see Lemma \ref{supermartingale}), we establish that $x^{k}\to x^*$ a.s. Similarly, we can obtain that $\operatorname{dist}_{\mathcal{X^*}}(x^{k})\to 0$ a.s.
	Furthermore,  Theorem \ref{THM-asym} $(\romannumeral2)$ and $(\romannumeral3)$ indicate Algorithm \ref{RIM} exhibits the variance reduction property. In fact, supposing $\mathbb{E}[x]$ is bounded for all $ x \in \mathbb{R}^n $, by definition, we have
	$$
	\mathbb{E} [ \|x-x^* \|^2_2] =\left\|\mathbb{E} \left[ x-x^{*} \right] \right\|_2^{2}
	+\mathbb{E} \left[ \|x-\mathbb{E}[x]\|_2^{2}\right],
	$$
	which implies that the convergence of $ \mathbb{E} [ \|x-x^*\|^2_2] $ leads to that of $\mathbb{E} \left[ \|x-\mathbb{E}[x]\|_2^{2}\right]$, namely, the reduction of variance. 
\end{remark}

\section{Special cases: Examples} 
\label{sect-examples}

This section provides a brief discussion on the choice of the probability space $(\Omega,\mathcal{F},\mathbf{P})$  in our method for recovering existing methods and developing new ones.
While this list is not exhaustive, it serves to illustrate the flexibility of our algorithm. 

\subsection{The generalized Picard iteration method}
\label{sect:5-1}

When the sampling space $\Omega=\{(A^\dagger)^\top\}$, i.e. $S=(A^\dagger)^\top$ with probability one, and using the fact that $A^\top (A^\dagger)^\top A^\dagger =A^\dagger$, the iteration scheme \eqref{one-step} of our method becomes
\begin{equation}
	\label{gpim}
	x^{k+1}=x^k-\alpha A^\dagger(Ax^k-B|x^k|-b).
\end{equation}
In particular, when $\alpha=1$, $m=n$ and $A$ is nonsingular, the iteration scheme \eqref{gpim} recovers the Picard iteration method \cite{rohn2014iterative}. Hence, we refer to \eqref{gpim} as the generalized Picard iteration method. 

In this case, the parameters in Theorem \ref{THM-asym} can be simplified as 
$H=(A^\dagger)^\top A^\dagger$, $\sigma_n(H^\frac{1}{2}A)=1$, and $\|H^\frac{1}{2}B\|_2=\|A^\dagger B\|_2$. If $\|A^\dagger B\|_2<1$, then according to Theorem \ref{THM-asym} 	$(\romannumeral2)$, the iteration scheme \eqref{gpim} with $\alpha=1$ has the following convergence property
$$
\|x^k-x^*\|_2\leq (\|A^\dagger B\|_2)^k\|x^0-x^*\|_2.
$$
The above convergence result recovers the result in \cite[Proposition 22]{zamani2023error} when $A$ is nonsingular.
We should note that the other convergence results in Theorem \ref{THM-asym} still hold, although we have only illustrated the result mentioned above.

\subsection{The  gradient descent method}
When the sampling space $\Omega=\{I_m\}$, i.e. $S=I_m$ with probability one, then  the iteration scheme \eqref{one-step} of our method becomes
\begin{equation}
	\label{ggdm}
	x^{k+1}=x^k-\alpha \frac{A^\top(Ax^k-B|x^k|-b)}{\|A\|^2_2}.
\end{equation}
In particular, when $B=0$, the iteration scheme \eqref{ggdm}  corresponds to the gradient descent method used to solve the least-squares problem. 

In this case, the parameters in Theorem \ref{THM-asym} can be simplified as 
$H=I/\|A\|^2_2$, $\sigma_n(H^\frac{1}{2}A)=\sigma_n(A)/\|A\|_2$, and $\|H^\frac{1}{2}B\|_2=\|B\|_2/\|A\|_2$. If $\sigma_n(A)>\|B\|_2$, then according to Theorem \ref{THM-asym} 	$(\romannumeral2)$, the iteration scheme \eqref{gpim} with $\alpha=1$ exhibits the following convergence property
$$
\|x^k-x^*\|_2^2\leq \left(1-\frac{\sigma^2_n(A)-\|B\|^2_2}{\|A\|^2_2}\right)^k\|x^0-x^*\|_2^2.
$$ 
When $B=0$, the above convergence property aligns with the convergence property of gradient descent for solving the least-squares problem, see e.g. \cite[Theorem 3.10]{bubeck2015convex}.

\subsection{Randomized Kaczmarz method for GAVE}\label{sect-5-3}
When the sampling space $ \Omega=\left\{e_i\right\}_{i=1}^m$ with $e_{i_k} $ being sampled with probability $ \operatorname{Prob}(i_k=i)= \frac{\|A_{i,:}\|^2_2}{\|A\|_F^2}$ at the $ k $-th iteration,  then  the iteration scheme \eqref{one-step} of our method becomes
\begin{equation}
	\label{rk-gave}
	x^{k+1}=x^k-\alpha \frac{A_{i_k,:}x^k-B_{i_k,:}|x^k|-b_{i_k}}{\|A_{i_k,:}\|^2_2}A_{i_k,:}^\top.
\end{equation}
Particularly, when $B=0$, the iteration scheme \eqref{rk-gave}  corresponds to the randomized Kaczmarz method for solving linear systems. 

In this case, the parameters in Theorem \ref{THM-asym} can be simplified as 
$H=I/\|A\|^2_F$, $\sigma_n(H^\frac{1}{2}A)=\sigma_n(A)/\|A\|_F$, and $\|H^\frac{1}{2}B\|_2=\|B\|_2/\|A\|_F$. If $\sigma_n(A)>\|B\|_2$, then according to Theorem \ref{THM-asym} 	$(\romannumeral2)$, the iteration scheme \eqref{gpim} with $\alpha=1$ exhibits the following convergence property
$$
\mathbb{E}[\|x^k-x^*\|_2^2]\leq \left(1-\frac{\sigma^2_n(A)-\|B\|^2_2}{\|A\|^2_F}\right)^k\|x^0-x^*\|_2^2.
$$ 
When $B=0$, the above convergence property coincides with the convergence property of the RK method.

\subsection{Randomized  block Kaczmarz method for GAVE}

When the sampling space $ \Omega=\left\{I_{:,\mathcal{J}}((A_{\mathcal{J},:})^\dagger)^\top\right\}_{\mathcal{J}\subset[m]}$, then  the iteration scheme \eqref{one-step} of our method becomes
\begin{equation}
	\label{rbk-gave}
	x^{k+1}=x^k-\alpha (A_{\mathcal{J}_k,:})^\dagger(A_{\mathcal{J}_k,:}x^k-B_{\mathcal{J}_k,:}|x^k|-b_{\mathcal{J}_k}).
\end{equation}
Particularly, when $B=0$, the iteration scheme \eqref{rbk-gave}  corresponds to the randomized block Kaczmarz (RBK) method for solving linear systems \cite{needell2014paved}. 
We briefly review the strategies for selecting the subset $\mathcal{J}_k$ for the RBK method \cite{needell2014paved,ferreira2024survey}. We define the number of blocks denoted by $t$ and  divide the rows of the matrix into $t$ subsets, creating a partition $\Gamma=\{\mathcal{J}_1,\ldots,\mathcal{J}_t\}$. 
Then the block  $\mathcal{J}_k$ can be chosen from the partition using one of two strategies: it can be chosen randomly from the partition independently of all previous choices, or it can be sampled without replacement, which Needell and Tropp  \cite{needell2014paved} found to be more effective. 
If $t= m$, that is, if the number of blocks is equal to the number of rows, each block consists of a single row and we recover the RK method.
The calculation of the pseudoinverse $(A_{\mathcal{J}_k,:})^\dagger$ in each iteration is computationally expensive. However, if the submatrix $(A_{\mathcal{J}_k,:})^\dagger$ is well-conditioned, we can use efficient algorithms like conjugate gradient for least-squares (CGLS) to calculate it.

To analyze the convergence of the RBK method, it is necessary to define specific quantities \cite{needell2014paved}. The row paving
$(t,\beta_1, \beta_2)$ of a matrix $A$ is a partition $\Gamma=\{\mathcal{J}_1,\ldots,\mathcal{J}_t\}$ that  satisfies
$$
\beta_1\leq \lambda_{\min}(A_{\mathcal{J},:}A^\top_{\mathcal{J},:}) \ \ \text{and} \ \ \lambda_{\max}(A_{\mathcal{J},:}A^\top_{\mathcal{J},:})\leq \beta_2, \ \  \forall \ \mathcal{J}\in\Gamma, 
$$
where $t$ indicates the size of the paving,  and  $\beta_1$ and $\beta_2$ denote the lower and upper paving bounds, respectively. 
Assume that $A$ is a matrix with full column rank and row paving $(t,\beta_1, \beta_2)$, and the index $i$ is selected with probability $ 1/t$.
Let us now consider the parameters  in Theorem \ref{THM-asym}. The parameter $H=\frac{1}{t}\sum_{i=1}^{t} I_{:,\mathcal{J}_i}((A_{\mathcal{J}_i,:})^\dagger)^\top (A_{\mathcal{J}_i,:})^\dagger(I_{:,\mathcal{J}_i})^\top$ 
satisfies $\frac{1}{t\beta_2 } I\preceq H\preceq \frac{1}{t\beta_1 }I$,  and consequently, $\sigma_n(H^\frac{1}{2}A)\geq\sigma_n(A)/\sqrt{t\beta_2}$, and $\|H^\frac{1}{2}B\|_2\leq\|B\|_2/\sqrt{t\beta_1}$.
If $\sigma_n(A)>\sqrt{\frac{\beta_2}{\beta_1}}\|B\|_2$, then according to Theorem \ref{THM-asym} 	$(\romannumeral2)$, the iteration scheme \eqref{rbk-gave} with $\alpha=1$ exhibits the following convergence property
$$
\mathbb{E}[\|x^k-x^*\|_2^2]\leq \left(1-\left(\frac{\sigma_n^2(A)}{t\beta_2}-\frac{\|B\|^2_2}{t\beta_1}\right)\right)^k\|x^0-x^*\|_2^2.
$$ 
When $B=0$, this convergence property aligns with that of the RBK method for solving consistent linear systems with full column rank coefficient matrices, as detailed in \cite[Theorem 1.2]{needell2014paved}.

In practice, the RBK method may encounter several challenges, including the computational expense at each iteration due to the necessity of applying the pseudoinverse to a vector, which is equivalent to solving a least-squares problem. Additionally, the method presents difficulties in  parallelization. To overcome these obstacles, the
randomized average block Kaczmarz (RABK) method was introduced \cite{necoara2019faster,moorman2021randomized,du2020randomized,zeng2023adaptive}.

\subsection{Randomized average block Kaczmarz method for GAVE}
\label{sub-section-rabk}

The randomized average block Kaczmarz (RABK) method is a block-parallel approach that computes multiple updates at each iteration \cite{necoara2019faster,moorman2021randomized,du2020randomized,zeng2023adaptive}.
Specifically,  consider the following partition of $[m]$
$$	
\begin{aligned}
	\mathcal{I}_i&=\left\{\varpi(j): j=(i-1)p+1,(i-1)p+2,\ldots,ip\right\}, i=1, 2, \ldots, t-1,
	\\
	\mathcal{I}_t&=\left\{\varpi(j): j=(t-1)p+1,(t-1)p+2,\ldots,m\right\}, \operatorname{card}(\mathcal{I}_t)\leq p,
\end{aligned}$$
where $\varpi$ is a uniform random permutation on $[m]$ and $p$ is the block size. 
We define $\|A\|_{\varpi,p}:=\sqrt{\sum_{i=1}^{t}\|A_{\mathcal{I}_i,:}\|_2^2}$ and
select an index $ i_k \in [t] $ with the probability  $\operatorname{Prob}(i_k=i)=\|A_{\mathcal{I}_i,:}\|^2_2/\|A\|^2_{\varpi,p}$, and then set $S_k=I_{:,\mathcal{I}_{i_k}}$. 
The iteration scheme \eqref{one-step} of our method becomes
\begin{equation}
	\label{rabk-gave}
	x^{k+1}=x^k-\alpha \frac{A_{\mathcal{I}_{i_k},:}^\top\left(A_{\mathcal{I}_{i_k},:}x^k-B_{\mathcal{I}_{i_k},:}|x^k|-b_{\mathcal{I}_{i_k}}\right)}{\|A_{\mathcal{I}_{i_k},:}\|^2_2}.
\end{equation}
Particularly, when $B=0$, the iteration scheme \eqref{rabk-gave}  corresponds to the randomized average block Kaczmarz (RABK) method for solving linear systems \cite{necoara2019faster,du2020randomized,moorman2021randomized,zeng2023adaptive}. 

In this case, the parameters in Theorem \ref{THM-asym} can be simplified as 
$H=1/\|A\|^2_{\varpi,p}$, $\sigma_n(H^\frac{1}{2}A)=\sigma_n(A)/\|A\|_{\varpi,p}$, and $\|H^\frac{1}{2}B\|_2=\|B\|_2/\|A\|_{\varpi,p}$. If $\sigma_n(A)>\|B\|_2$, then according to Theorem \ref{THM-asym} 	$(\romannumeral2)$, the iteration scheme \eqref{gpim} with $\alpha=1$ exhibits the following convergence property
\begin{equation}
	\label{RABK-convergence}
	\mathbb{E}[\|x^k-x^*\|_2^2]\leq \left(1-\frac{\sigma^2_n(A)-\|B\|^2_2}{\|A\|^2_{\varpi,p}}\right)^k\|x^0-x^*\|_2^2.
\end{equation}
If $p=1$, meaning that each block only contains a single row, then the iterative scheme \eqref{rabk-gave} recovers the RK method \eqref{rk-gave} for GAVE, and the convergence result stated above aligns with the result established in Section \ref{sect-5-3}.

We now make a comparison between the cases where $p=1$ and $p=m$. For convenience, we assume that $\sigma^2_n(A)-\|B\|^2_2=1$.  The convergence factors for the cases $p=1$  and $p=m$ are  $1-\frac{1}{\|A\|^2_F}$ and $1-\frac{1}{\|A\|^2_2}$, respectively. Since the computational cost for the case $p=m$ at each step is about $m$-times as expensive as that for the case $p=1$, 
we can turn this comparison into a comparison between $\left(1-\frac{1}{\|A\|^2_F}\right)^m$ and $ 1-\frac{1}{\|A\|^2_2}$.
Since $\|A\|^2_F\leq m\|A\|_2^2$, it follows that 
$$\left(1-\frac{1}{\|A\|^2_F}\right)^m\leq \left(1-\frac{1}{m\|A\|^2_2}\right)^m< \operatorname{exp}\left(-\frac{1}{\|A\|^2_2}\right)<1-\frac{1}{\|A\|^2_2}.$$
This suggests that, theoretically, the RABK method with $p=1$ is more efficient than the RABK method with $p=m$ for solving the GAVE \eqref{GAVE}.
However, we note that one can use the parallelization technique to speed-up  the iteration scheme \eqref{rabk-gave} in terms of the total running time.


\subsection{Other types of sketches}
\label{sect-sketch}

In addition to the previously discussed sketching matrices, we now introduce several other types of sketching matrices that can be employed in Algorithm \ref{RIM}. These alternative techniques are widely used in randomized numerical linear algebra and optimization, each offering unique advantages that make them particularly suitable for specific problems. For a more comprehensive discussion of additional sketching methods, we refer the reader to e.g., \cite[Sections 8 and 9]{martinsson2020randomized}.

{\bf Uniform sampling and CountSketch:} 
We consider the uniform sampling of \(p\) unique indices that form the set \(\mathcal{J}\), where \(\mathcal{J}\subset[m]\) and \(\operatorname{card}(\mathcal{J})=p\) for all samplings, with \(p\) representing the block size. It is evident that the total number of possible choices for \(\mathcal{J}\) is given by \(\binom{m}{p}\), and the probability of selecting any particular \(\mathcal{J}\) is \(\operatorname{Prob}(\mathcal{J})=1/\binom{m}{p}\).  The uniform sketching matrix is defined as \(S_{US}=I_{:,\mathcal{J}}\). 

The CountSketch, which originated from the streaming data literature \cite{charikar2004finding,cormode2005improved} and was popularized as a matrix sketching tool by \cite{clarkson2017low}, is defined as  
$
S_{CS}^\top=DI_{\mathcal{J},:},
$
where \(D\in\mathbb{R}^{p\times p}\) is a diagonal matrix with elements sampled uniformly from \(\{-1,1\}\) and  the set \(\mathcal{J}\) is  chosen by uniformly sampling \(p\) elements from $[m]$. Since $S_{US}S_{US}^\top=S_{CS}S_{CS}^\top$ and $\|S_{US}^\top A\|_2=\|S_{CS}^\top A\|_2$, it follows that the iteration scheme \eqref{one-step} in Algorithm \ref{RIM} for both uniform sampling and CountSketch are identical. In these cases, we refer to Algorithm \ref{RIM} as the \emph{randomized iterative method} with \emph{CountSketch} (RIMCS).


{\bf Gaussian sketch:} A Gaussian sketch is a random matrix \(S\in\mathbb{R}^{m\times p}\), where each element is independently and identically distributed (i.i.d.) according to the standard Gaussian distribution. Forming the sketch \(S^\top A\) for a dense matrix \(A\) typically requires \(O(mnp)\) flops, which can be computationally expensive. However, in practice, the actual running time can be significantly reduced through parallelized computation or when dealing with sparse matrices \(A\). For a detailed discussion of the practical advantages of Gaussian sketch, one may refer to \cite{meng2014lsrn}.  In this case, Algorithm \ref{RIM} yields the \emph{randomized iterative method} with \emph{Gaussian sketch} (RIMGS).


{\bf Subsampled randomized Hadamard transform (SRHT):}  The SRHT \cite{woolfe2008fast,tropp2011improved} is defined as a matrix \(S\in\mathbb{R}^{m\times p}\) such that
$
S^\top=\sqrt{\frac{m}{p}}I_{\mathcal{J},:}H_m D,
$
where \(D\in\mathbb{R}^{m\times m}\) is a diagonal matrix with elements sampled uniformly from \(\{-1,1\}\), \(H_m\in\mathbb{R}^{m\times m}\) is the Hadamard matrix\footnote{The Hadamard transform is defined for \(m=2^q\) for some positive integer \(q\). If \(m\) is not a power of $2$, a standard practice is to pad the data matrix with \(2^{\lceil\log_2(m)\rceil}-m\) additional rows of zeros.} of order \(m\), and the set \(\mathcal{J}\) is  chosen by uniformly sampling \(p\) elements from $[m]$. We note that the SRHT can be computed using $O(mn\log p)$ flops via  the fast Walsh--Hadamard transform \cite{fino1976unified}. In this context, Algorithm \ref{RIM} yields the \emph{randomized iterative method} with \emph{SRHT} (RIMSRHT).


\section{Numerical experiments}
\label{sect:6}
In this section, we study the computational behavior of the proposed randomized iterative method. We also compare our algorithms with some of the state-of-the-art methods, namely, the generalized Newton method (GNM) \cite{mangasarian2009generalized,hu2011generalized},  the Picard iteration method (PIM) \cite{rohn2014iterative}, the successive linearization algorithm (SLA) \cite{mangasarian2007absolute}, and the method of alternating projections (MAP) \cite{alcantara2023method}.

All the methods are implemented in  MATLAB R2022a for Windows $11$ on a desktop PC with Intel(R) Core(TM) i7-1360P CPU @ 2.20GHz  and 32 GB memory. The code to reproduce our results can be found at \href{https://github.com/xiejx-math/GAVE-codes}{https://github.com/xiejx-math/GAVE-codes}.

We have previously detailed the PIM  in Section \ref{sect:5-1}. Next, we briefly describe GNM, SLA, and MAP, respectively.

\begin{itemize}
	\item[(\romannumeral1)] {\bf Generalized Newton method (GNM)} \cite{mangasarian2009generalized,hu2011generalized}.
	This algorithm is aimed at solving the GAVE \eqref{GAVE} with $m=n$ and the iterations are given by
	$$
	x^{k+1}=(A-BD^k)^{-1}b,
	$$
	where $D^k=\operatorname{diag}(\operatorname{sign}(x^k_1),\ldots,\operatorname{sign}(x^k_n))$. The iterations are derived by applying the semismooth Newton method in solving the equation $Ax-B|x|=b$.
	As in \cite{mangasarian2009generalized}, we use the MATLAB's backslash operator ``\textbackslash '' to obtain the iterates. 
	
	\item[(\romannumeral2)]  {\bf Successive linearization algorithm (SLA)} \cite{mangasarian2007absolute}.
	Given an initial point $(x^0,t^0,s^0)\in\mathbb{R}^n\times\mathbb{R}^n\times \mathbb{R}^m$ and $\varepsilon>0$, we solve the linear programming 
	\begin{equation}
		\label{sla}
		\begin{aligned}
			\min_{(x,t,s)\in\mathbb{R}^n\times\mathbb{R}^n\times \mathbb{R}^m} \ \ &\varepsilon\sum_{i=1}^{n} (-\operatorname{sign}(x_i^k)x_i+t_i)+\sum_{i=1}^{m}s_i\\
			\text{s.t.} \ \ & -s\leq Ax-Bt-b\leq s,\\
			&  -t\leq x\leq t,
		\end{aligned}
	\end{equation}
	and call its solution $(x^{k+1},t^{k+1},s^{k+1})$. To solve \eqref{sla},
	we use the MATLAB function {\tt linprog}. 
	
	\item[(\romannumeral3)] {\bf Method of alternating projections (MAP)} \cite{alcantara2023method}.
	Let $w:=(u,v)\in\mathbb{R}^n\times\mathbb{R}^n$ and consider the following feasibility problem
	\begin{equation}
		\label{feaprob}
		\mbox{Find} \ w  \in C_1\mathop{\cap} C_2,
	\end{equation}
	where
	$$
	C_1:=\{w:Tw=\sqrt{2}b\}, 
	C_2:=\{w: u\geq 0,v\geq0, \ \text{and} \ \langle u,v \rangle=0\},
	$$ 
	and the matrix $T$ is defined as
	$$
	T:=[A-B \ \ -A-B]\in\mathbb{R}^{m\times 2n}.
	$$
	In \cite{alcantara2023method}, the authors showed that the GAVE \eqref{GAVE} is equivalent to the feasibility problem \eqref{feaprob}, i.e., if $(u^*,v^*)$ solves \eqref{feaprob}, then $x^*=\frac{1}{\sqrt{2}}(u^*-v^*)$ solves the GAVE \eqref{GAVE}. The method of alternating projections (MAP) can be employed to solve \eqref{feaprob}. Particularly,
	starting from any initial point $w^0$, the MAP iterates with the format
	$$w^{k+1}\in (P_{C_1} P_{C_2})(w^k).$$
	Here for any $w\in\mathbb{R}^{2n}$, we have
	$$
	P_{C_1}(w)=w-T^\dagger(Tw-\sqrt{2}b)
	$$
	and $z\in P_{C_2}(w)$ if and only if for each $i=1,\ldots,n$,
	$$
	(z_i,z_{n+i}) \in\left\{
	\begin{array}{ll}
		\{(0,(v_i)_+)\}, & u_i<v_i,
		\\
		\{((u_i)_+,0)\}, & u_i>v_i,
		\\
		\{(0,(v_i)_+),((u_i)_+,0)\}, & u_i=v_i.
	\end{array}
	\right.
	$$
\end{itemize}

In our implementations, all computations are initialized with $x^0=0$.  The computations are terminated once the relative solution error (RSE), defined as
$\text{RSE}=\|x^k-x^*\|^2_2/\|x^*\|^2_2$, or
the relative residual error (RRE), defined as
$\text{RRE}=\|Ax^k-B|x^k|-b\|^2_2/\|b\|^2_2$,  is less than a specific error tolerance.
For our algorithms, we set $\alpha=1$ and for the SLA, we set $\varepsilon=1$. All the results below are averaged over $20$ trials.

\subsection{Comparison of different types of sketches}
\label{sect-comp-sketh}
In this subsection, we compare the performance of different sketching methods presented in Section \ref{sect-examples}. In particular, we focus mainly on the evaluation of the performance of  RABK, RIMCS, RIMGS, and RIMSRHT  detailed in Subsections \ref{sub-section-rabk} and \ref{sect-sketch}.

We generate the coefficient matrices \(A\) and \(B\) as follows. 
Given parameters \(m, n, a_{\min}, b_{\max}\), and  \(\kappa_A, \kappa_B\geq1\), let \(r = \min\{m, n\}\). We construct matrices \(A = U_1 D V_1^\top\) and \(B = U_2 P V_2^\top\), where \(U_1, U_2 \in \mathbb{R}^{m \times r}\) and \(V_1, V_2 \in \mathbb{R}^{n \times r}\) are column-orthogonal matrices. The diagonal matrices \(D\) and \(P\) have entries defined as:
\[
D_{i,i} = a_{\min} + \frac{i-1}{r-1} (\kappa_A - 1) a_{\min}, \quad
P_{i,i} = \frac{b_{\max}}{\kappa_B} + \frac{i-1}{r-1} \left(1 - \frac{1}{\kappa_B}\right) b_{\max}.
\]
This construction ensures that \(\sigma_r(A) = a_{\min}\), \(\|B\|_2 = b_{\max}\), and the condition numbers of matrices \(A\) and \(B\) are \(\kappa_A\) and \(\kappa_B\), respectively. Moreover, if \(m \geq n\) and \(a_{\min} > b_{\max}\), we can ensure the uniqueness of the solution to the GAVE (see Theorem \ref{THM-suf}). Using MATLAB notation, we generate the column-orthogonal matrices with the following commands: {\tt [U,$\sim$]=qr(randn(m,r),0)} and {\tt [V,$\sim$]=qr(randn(n,r),0)}. The exact solution is generated by \(x^* = {\tt randn(n,1)}\), and we then set \(b = Ax^* - B|x^*|\). In our test, we set \(a_{\min} = 2\) and \(b_{\max} = 1\).

Figure \ref{figureR1} presents the computational time (CPU) and the number of iterations for each method.  The bold line illustrates the median value derived from $20$ trials. The lightly shaded area signifies the  range from the minimum to the maximum values, while the darker shaded one indicates the data lying between the $25$-th and $75$-th quantiles. It can be seen from Figure \ref{figureR1} that  all four algorithms require the same number of iterations, which increases as the condition number of matrix \(A\) rises.
Meanwhile, the condition number of matrix $B$ has only a marginal influence on the iteration count. This suggests that the algorithms are more sensitive to the conditioning of matrix $A$ than to that of matrix $B$. These findings align with the convergence bound \eqref{RABK-convergence}, which is primarily dependent on the condition number of $A$ and only weakly influenced by the condition number of $B$.


From Figure \ref{figureR1}, it can also be observed that RIMGS and RIMSRHT demonstrate more CPU time compared to the other algorithms due to their higher computational cost per iteration. Although RABK and RIMCS have the same computational cost per iteration, RABK outperforms RIMCS in terms of the CPU time. This is attributed to RIMCS encountering additional computational overhead from repeatedly extracting rows from matrices $A$ and $B$.
In contrast, RABK avoids this issue by pre-storing the submatrices of $A$ and $B$ based on the initial partition, thereby eliminating the need for repeated row extractions.
In fact, it has been observed that the RABK method consistently outperforms the other three methods regarding CPU time throughout our experiments. Consequently, we will focus our subsequent tests on the RABK method.

\begin{figure}[tbhp]
	\centering
	\begin{tabular}{cc}
		\includegraphics[width=0.31\linewidth]{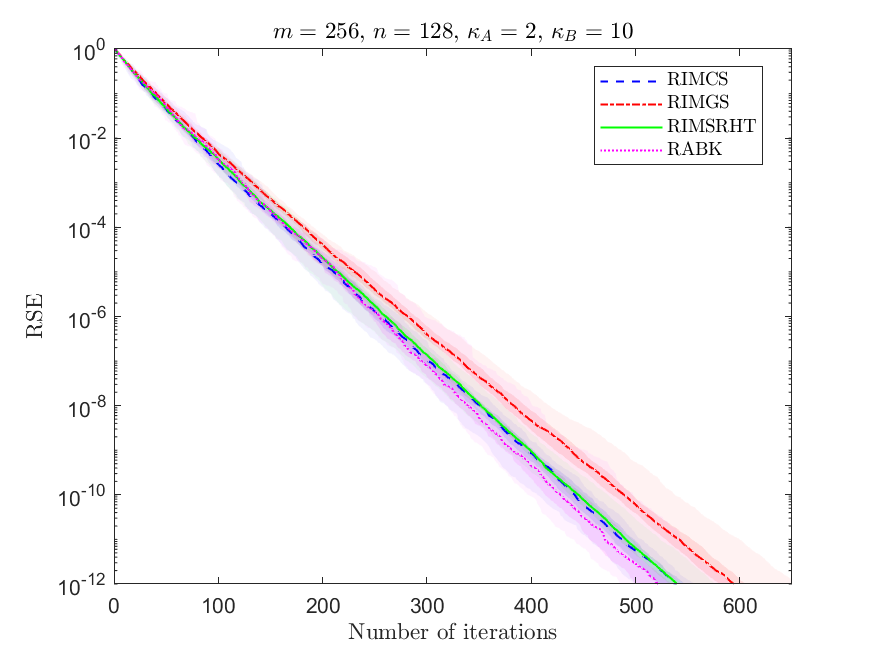}
		\includegraphics[width=0.31\linewidth]{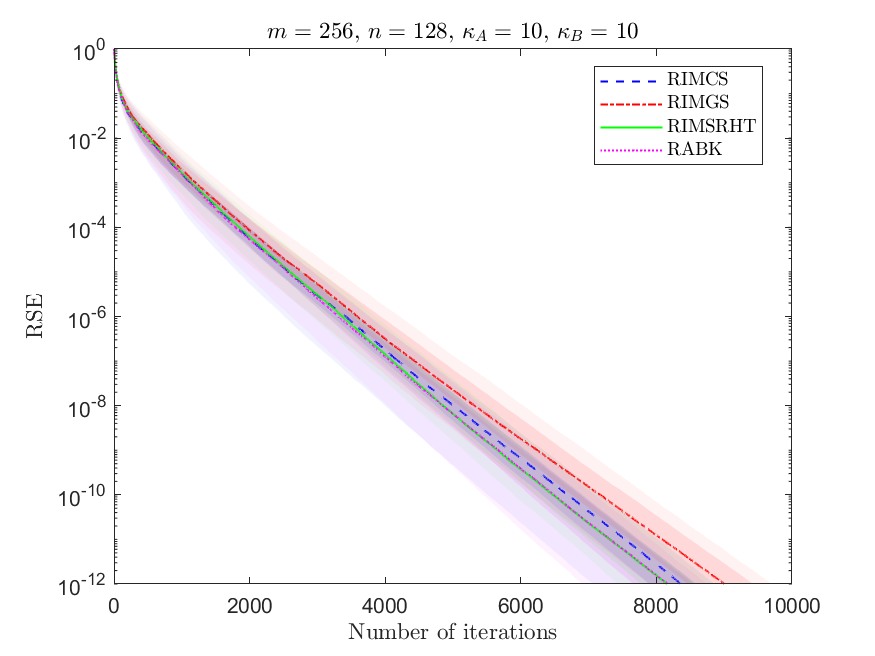}
		\includegraphics[width=0.31\linewidth]{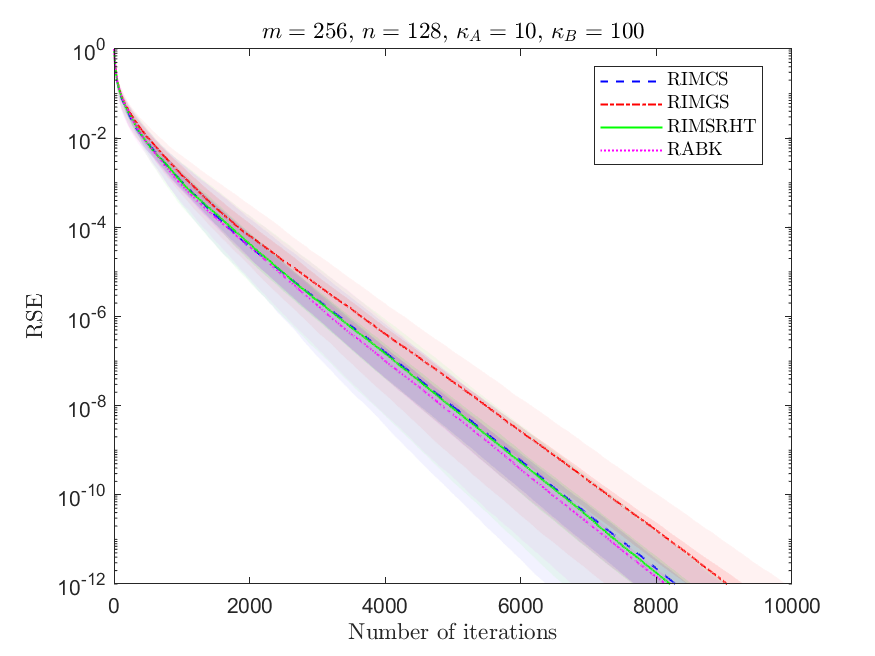}\\
		\includegraphics[width=0.31\linewidth]{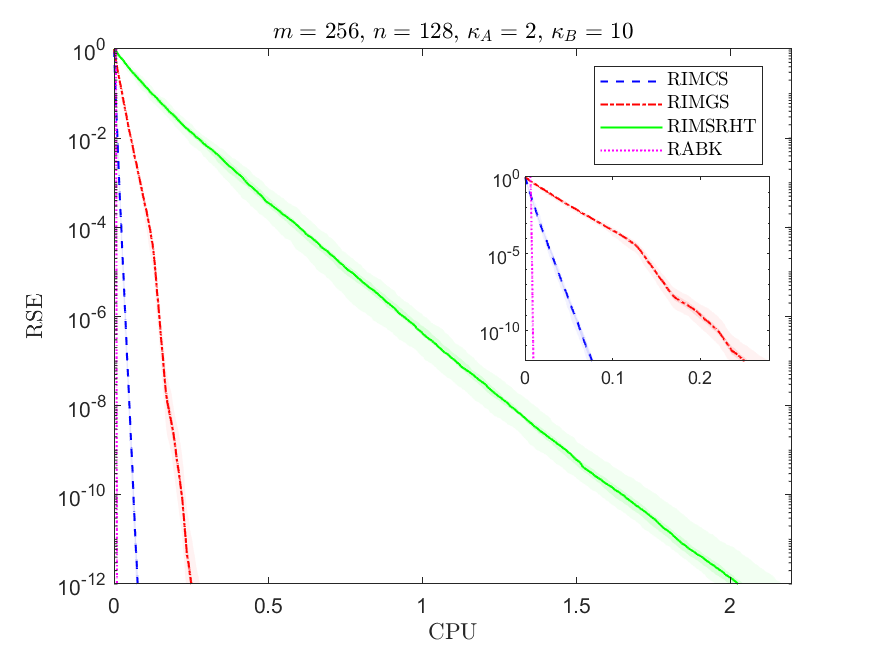}
		\includegraphics[width=0.31\linewidth]{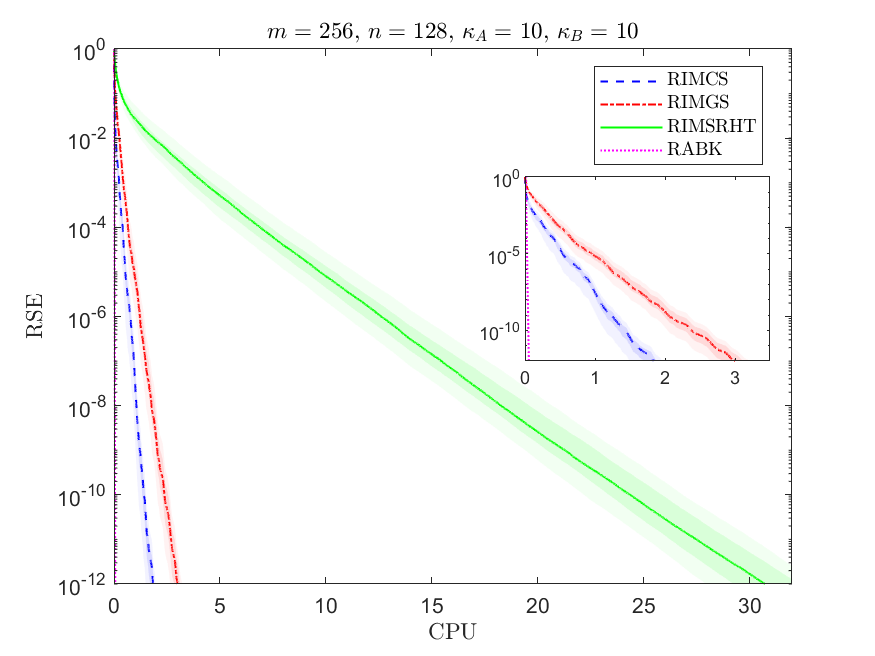}
		\includegraphics[width=0.31\linewidth]{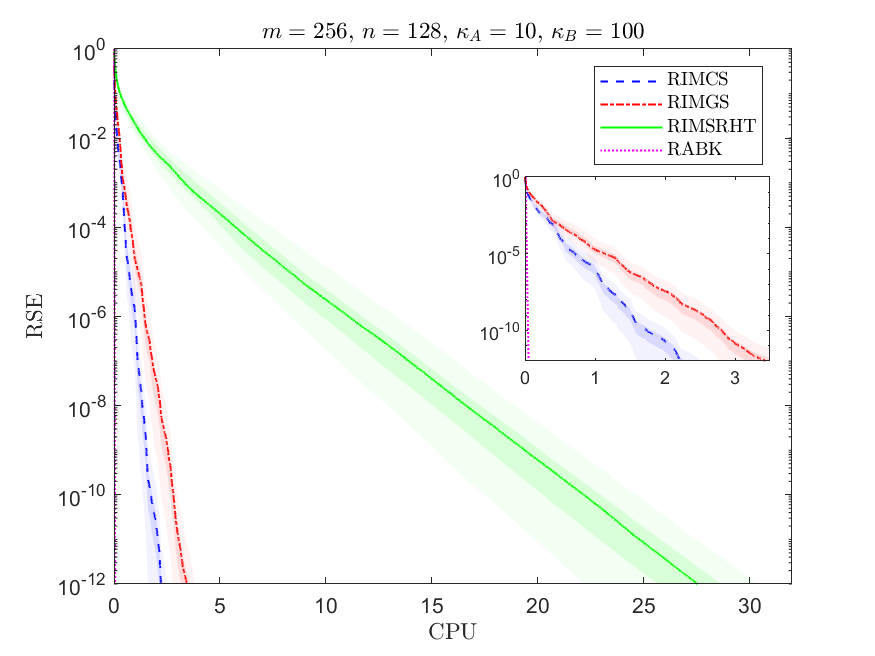}
	\end{tabular}
	\caption{Comparison of  RIMCS, RIMGS,  RIMSRHT, and RABK. We set $p=10$. Figures depict the evolution of RSE with respect to the number of  iterations (top) and the CPU time (bottom). The title of each plot indicates the values of $m,n,\kappa_A$, and $\kappa_B$.}
	\label{figureR1}
\end{figure}

Finally, we explore the influence of the block size $p$ on the convergence of the RABK method in solving the GAVE \eqref{GAVE}  with randomly generated matrices. The performance of the algorithms is measured in both the computing time (CPU) and 
the number of full iterations $(k\cdot\frac{p}{m})$, which  maintains a uniform count of operations for a single pass through the rows of \(A\) across all algorithms. In this experiment, we fix $m=512$ and set $n$ to be $128,256$,  and $512$.
The results are presented in Figure \ref{figure1}. As shown in Figure \ref{figure1}, increasing the value of \( p \) leads to a rise in the number of full iterations. This indicates that smaller values of \( p \) enhance the performance of the RABK method concerning the number of full iterations, which is consistent with the discussion in Section \ref{sub-section-rabk}. However, variations in CPU time for different values of \( p \) are relatively minor, with some instances showing an increase followed by a slight decrease. This behavior is due to MATLAB performing automatic multithreading when calculating matrix-vector products, which is the bottleneck in block sampling-based methods. Consequently, despite the higher number of full iterations required by the RABK method, its performance in terms of wall-clock time remains comparable to other methods.
In the subsequent tests, we would take $p=1$.
We note that when \(p=1\), the RABK method essentially becomes the RK method \eqref{rk-gave} for GAVE.

\begin{figure}[tbhp]
	\centering
	\begin{tabular}{cc}
		\includegraphics[width=0.31\linewidth]{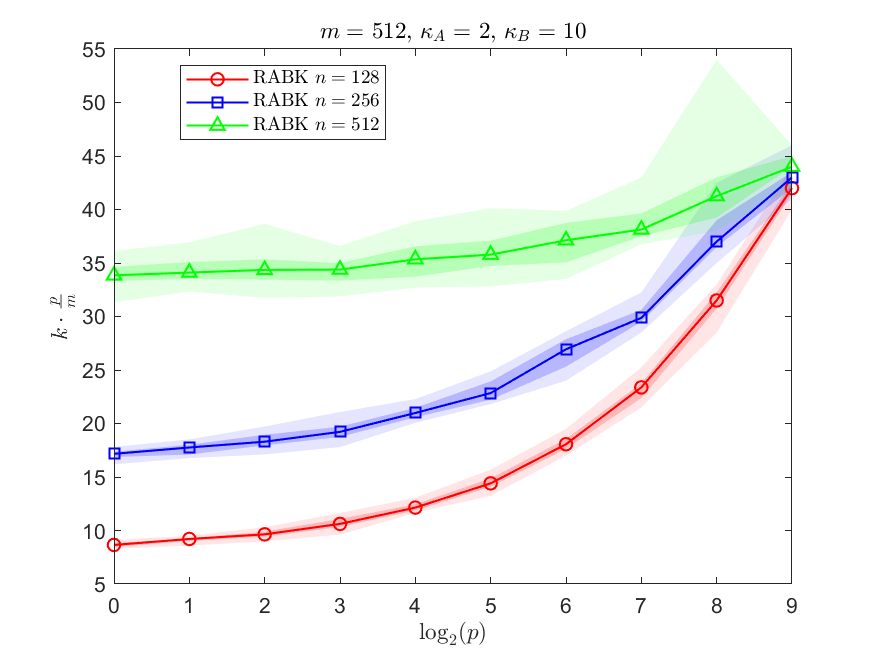}
		\includegraphics[width=0.31\linewidth]{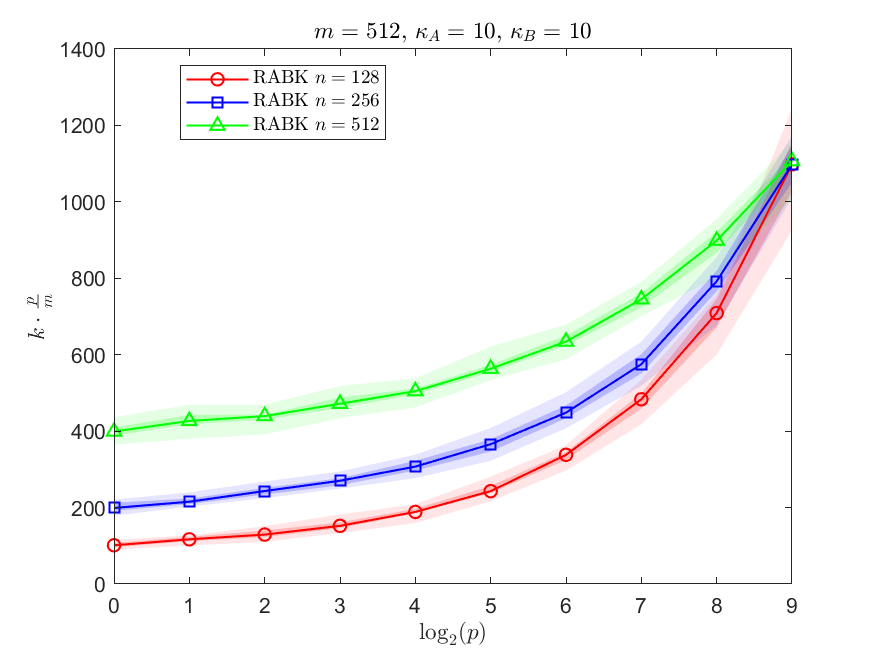}
		\includegraphics[width=0.31\linewidth]{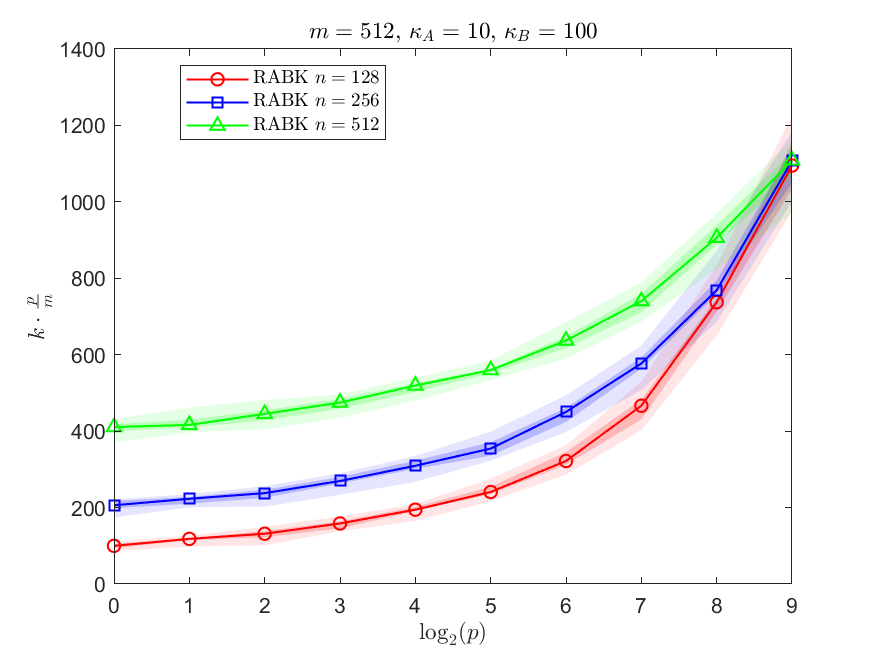}\\
		\includegraphics[width=0.31\linewidth]{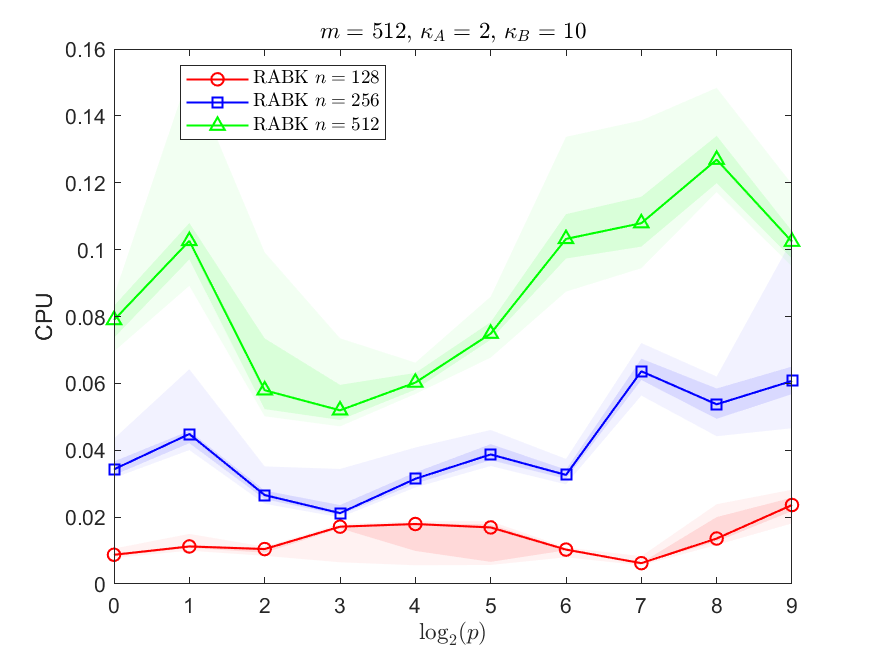}
		\includegraphics[width=0.31\linewidth]{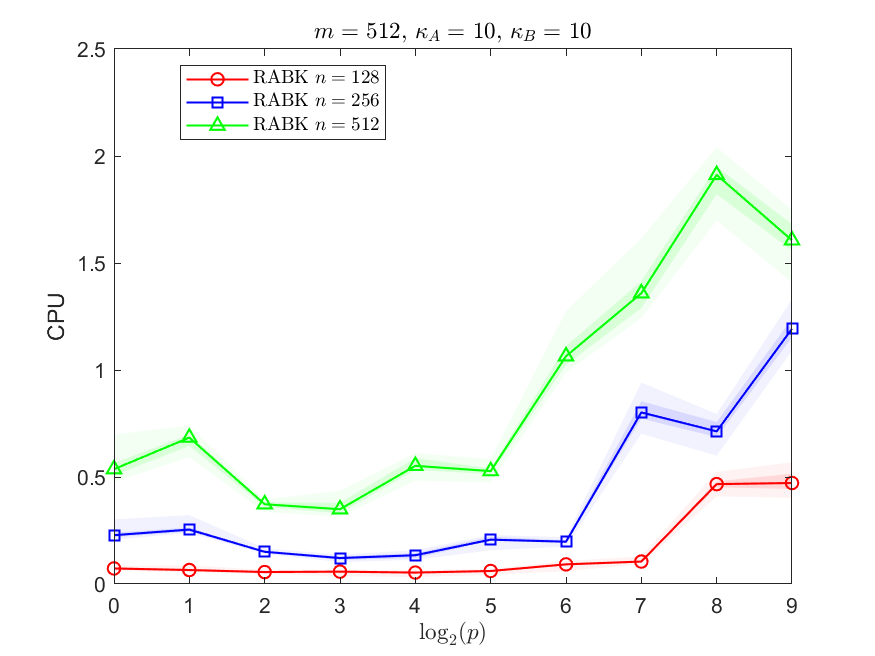}
		\includegraphics[width=0.31\linewidth]{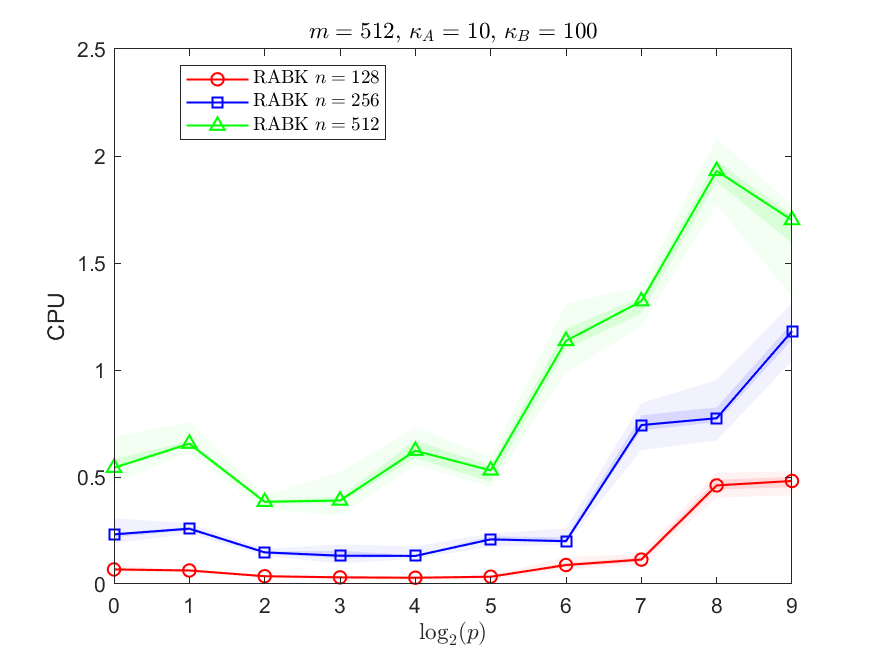}
	\end{tabular}
	\caption{Figures depict the evolution of the number of full iterations (top) and the CPU time (bottom) with respect to the block size $p$.  We fix $m=512$ and set $n$ to be $128,256$,  and $512$. All computations are terminated once  $\operatorname{RSE}<10^{-12}$.}
	\label{figure1}
\end{figure}

\subsection{Comparison to some existing methods}
In this subsection, we compare the performance of RABK with GNM, PIM, SLA, and MAP. For the case where \( m = n \), we compare RABK with GNM, PIM, and MAP, excluding SLA due to its extensive computational time requirements for these problems. For the case where \( m \neq n \), RABK is compared solely with SLA and MAP, as GNM and PIM are only capable of handling square matrices. 
The computations are terminated once the RSE is less than $10^{-12}$. Similar to Subsetion \ref{sect-comp-sketh}, the exact solution is generated by \(x^* = {\tt randn(n,1)}\). We set \(b = Ax^* - B|x^*|\),  \(a_{\min} = 2\), \(b_{\max} = 1\), and $p=1$.

In Figure \ref{figure2}, we present the CPU time comparisons for PIM, GNM, MAP, and RABK methods using square coefficient matrices, i.e., $m=n$. The results reveal that for relatively small values of $n$, the GNM, PIM, and MAP exhibit superior performance compared to the RABK method. However, as  $n$ increases, the performance of the RABK method gradually improves, eventually surpassing the other algorithms and emerging as the most efficient method. 
Figure \ref{figure3} illustrates the performance comparison among SLA, MAP, and RABK methods when applied to non-square coefficient matrices. The results demonstrate the efficiency of the RABK method.



\begin{figure}[tbhp]
	\centering
	\begin{tabular}{cc}
		\includegraphics[width=0.31\linewidth]{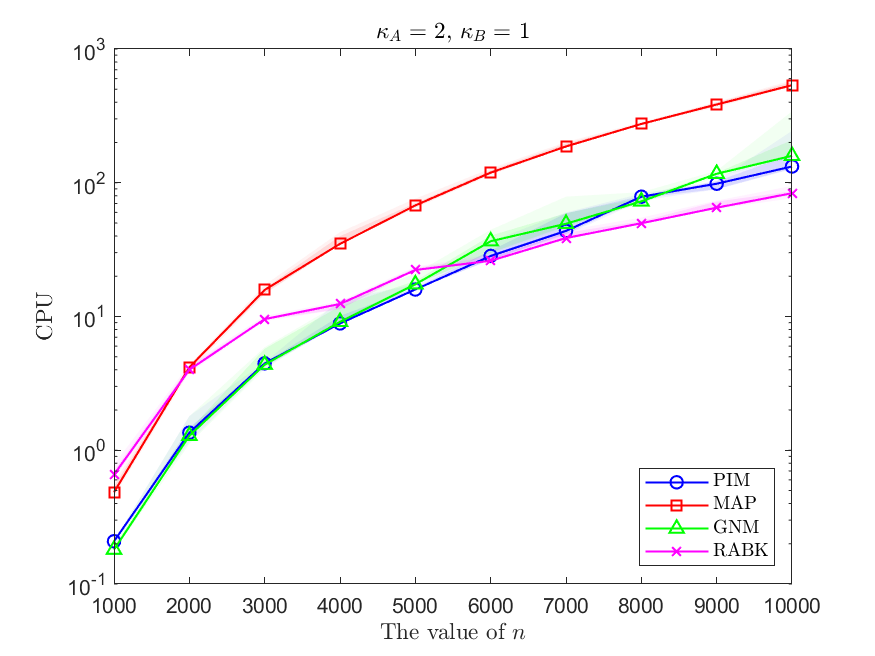}
		\includegraphics[width=0.31\linewidth]{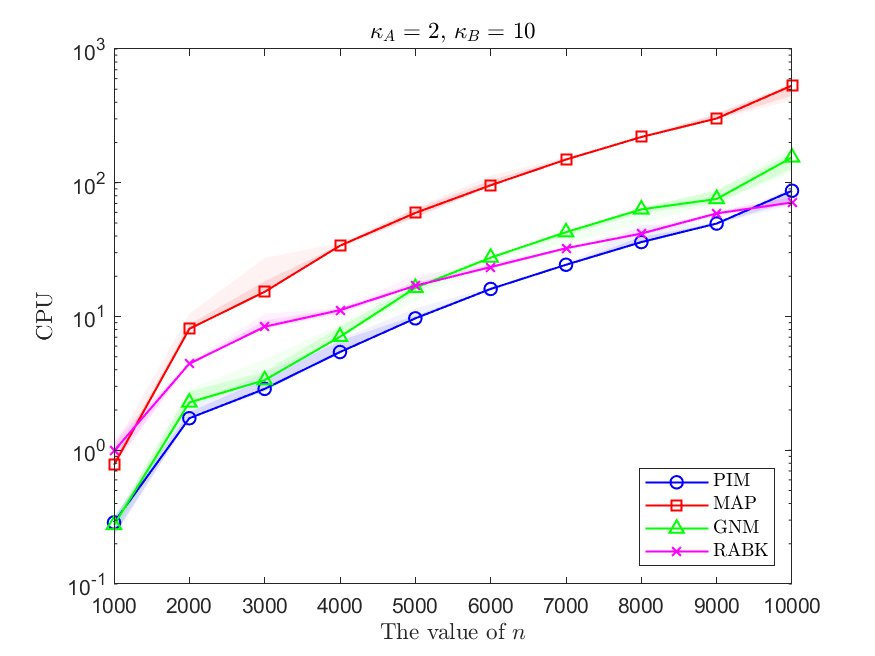}
		\includegraphics[width=0.31\linewidth]{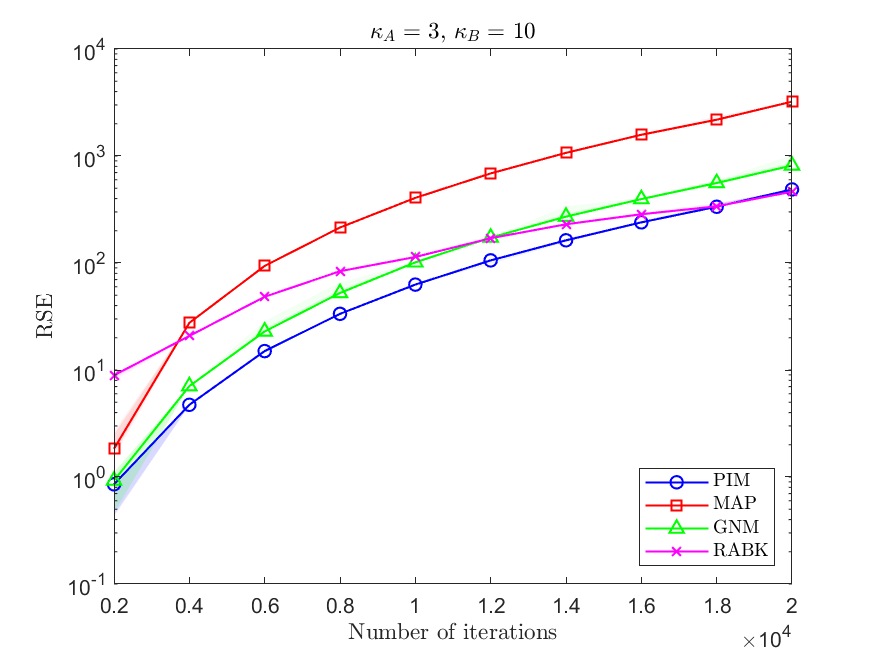}
	\end{tabular}
	\caption{
		Figures depict the evolution of CPU time vs the increasing dimensions of the coefficient matrices. We have $m=n$ and the title of each plot indicates the values of $\kappa_A$ and $\kappa_B$.  }
	\label{figure2}
\end{figure}

\begin{figure}[tbhp]
	\centering
	\begin{tabular}{cc}
		\includegraphics[width=0.31\linewidth]{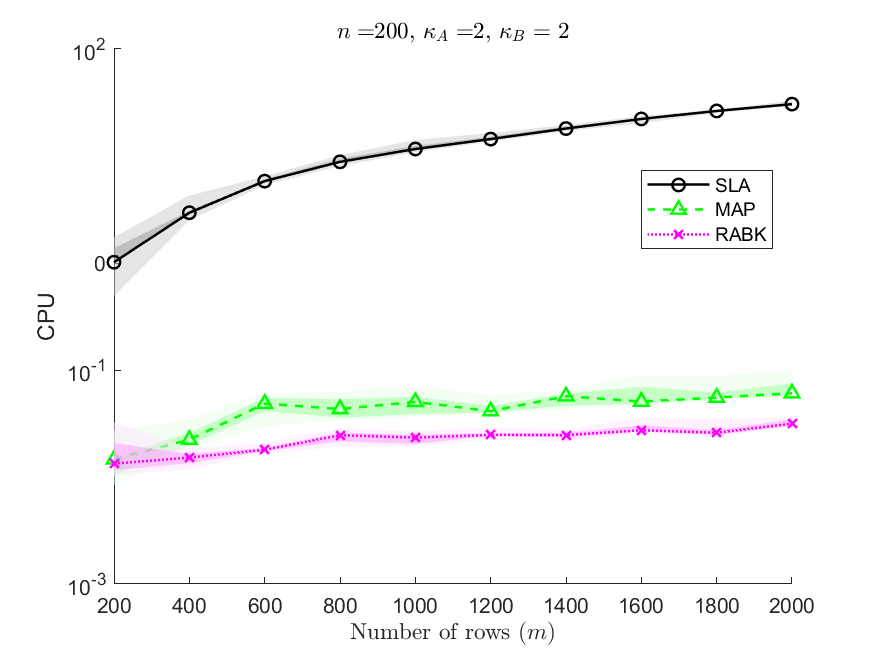}
		\includegraphics[width=0.31\linewidth]{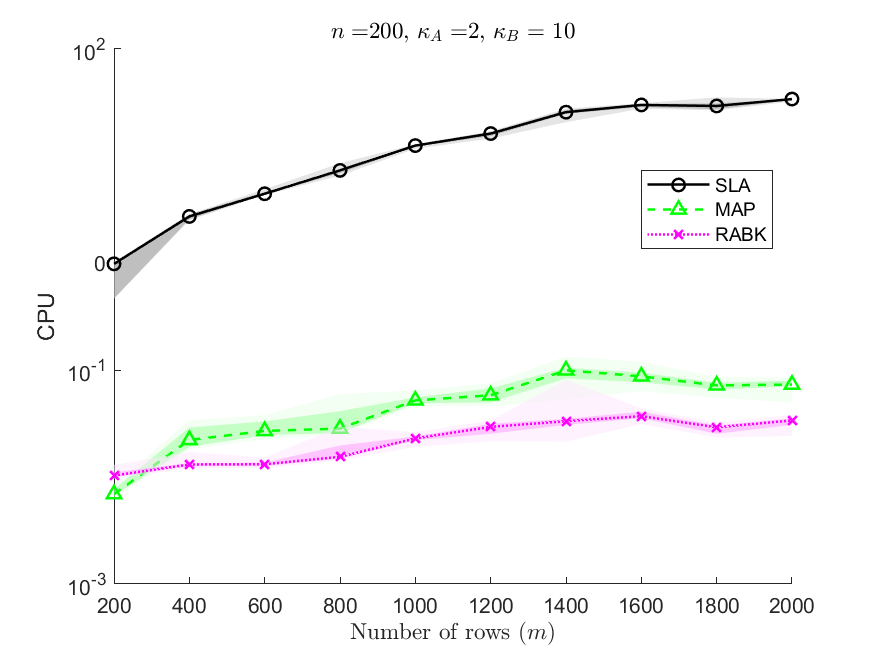}
		\includegraphics[width=0.31\linewidth]{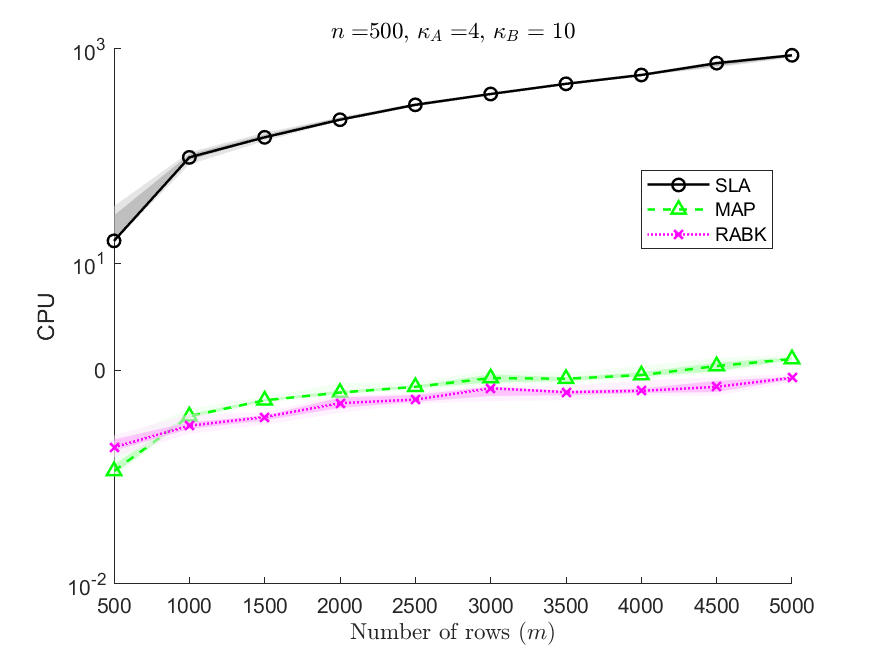}
	\end{tabular}
	\caption{Comparison of SLA, MAP, and RABK with non-square coefficient matrices. Figures depict the CPU time (in seconds) vs increasing number of rows. The title of each plot indicates the values of $n,\kappa_A$, and $\kappa_B$. }
	\label{figure3}
\end{figure}

\subsection{Ridge Regression}
Ridge regression is a popular parameter estimation method used to address the collinearity problem frequently arising in multiple linear regression \cite{mcdonald2009ridge,daniilidis2024solving}.
We consider an asymmetric ridge regression of the form:
$$
\min_{x\in\mathbb{R}^n} h(x)+\sum_{i=1}^{n}\left(\lambda_i\max\{x_i,0\}^2+\mu_i\max\{-x_i,0\}^2 \right),
$$ 
where the penalty parameters  $\lambda_i$ and $\mu_i$ satisfy $\lambda_i\neq\mu_i, i=1,\ldots,n$. We note that the case $\lambda_i=\mu_i=\lambda$ for every $i$ corresponds to the classical ridge regression, which will not be considered here. The case $\lambda_i=0$ for all $i$ and $\mu_i>0$ corresponds to a penalization of the negativity of the coefficients, promoting  solutions with positive coefficients. The necessary condition for optimality is given by:
$$
\nabla h(x)+2\lambda \circ\max\{x,0\}-2\mu\circ\max\{-x,0\}=0, 
$$
where $\circ$  denotes the componentwise (or Hadamard) product. Noting that $2\max\{x,0\}=|x|+x$ and $2\max\{-x,0\}=|x|-x$, we end up with the following problem
\begin{equation}\label{first-RR}
	\nabla h(x)+(\mu+\lambda)\circ x-(\mu-\lambda)\circ|x|=0.
\end{equation}
If we set $\lambda_i=\bar{\lambda}$ and $\mu_i=\bar{\mu}$ for every $i\in[n]$ and consider the loss function $h(x)=\frac{1}{2}\|Lx-c\|^2_2$, where $L\in\mathbb{R}^{m\times n}$ and $c\in\mathbb{R}^m$, then \eqref{first-RR} becomes the following GAVE problem
$$
\left(L^\top L+(\bar{\mu}+\bar{\lambda})I\right) x-(\bar{\mu}-\bar{\lambda})|x|=L^\top c.
$$ 

Figure \ref{figure4} illustrates our experimental results for varying values of $\bar{\lambda}$, and $\bar{\mu}$. The matrix $L$ and vector $c$ are randomly generated with values in $[-5,5]$. The computations are terminated once the RRE is less than $10^{-12}$ and set $p=1$ for the RABK method. 
The performance of the algorithms is measured in the computing time (CPU) and the number of full iterations. From the iteration schemes of PIM, GNM, and MAP, it can be seen that one iteration of PIM and GNM corresponds to one full iteration, while one iteration of MAP corresponds to two full iterations. It can be observed from Figure \ref{figure4} that RABK outperforms the other methods in terms of the number of full iterations.  However, in terms of CPU time, both PIM and GNM outperform RABK, even though we only show the CPU time taken by RABK to execute \eqref{rabk-gave}. This is because the parallel processing capability of  MATLAB can help alleviate the computational costs associated with computing the inverse of a matrix.

\begin{figure}[tbhp]
	\centering
	\begin{tabular}{cc}
		\includegraphics[width=0.31\linewidth]{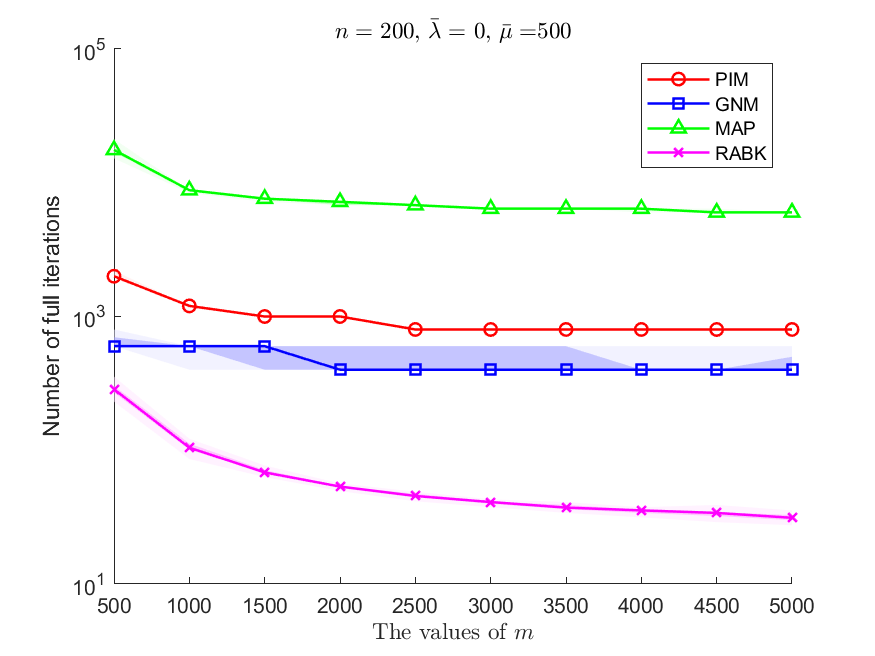}
		\includegraphics[width=0.31\linewidth]{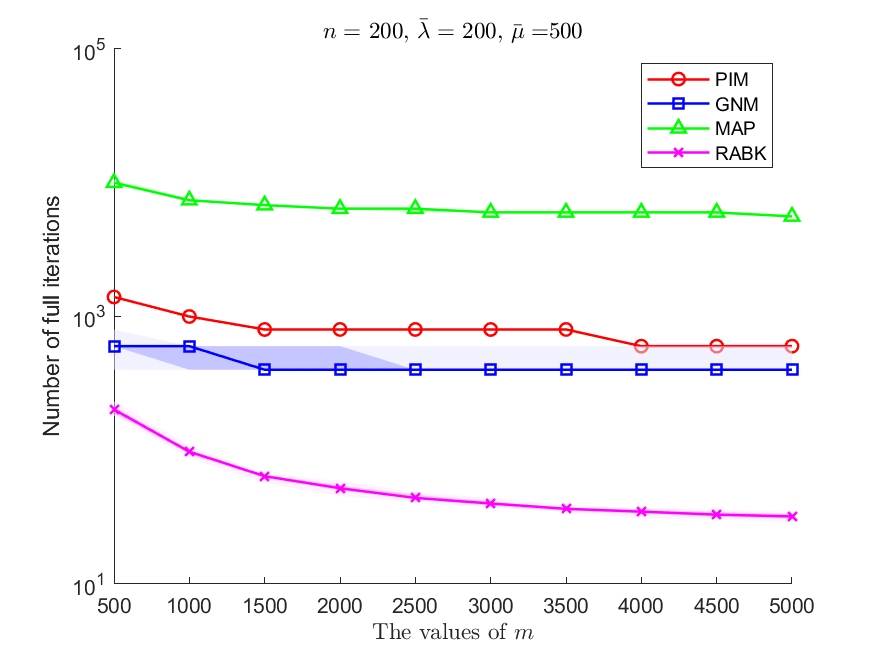}
		\includegraphics[width=0.31\linewidth]{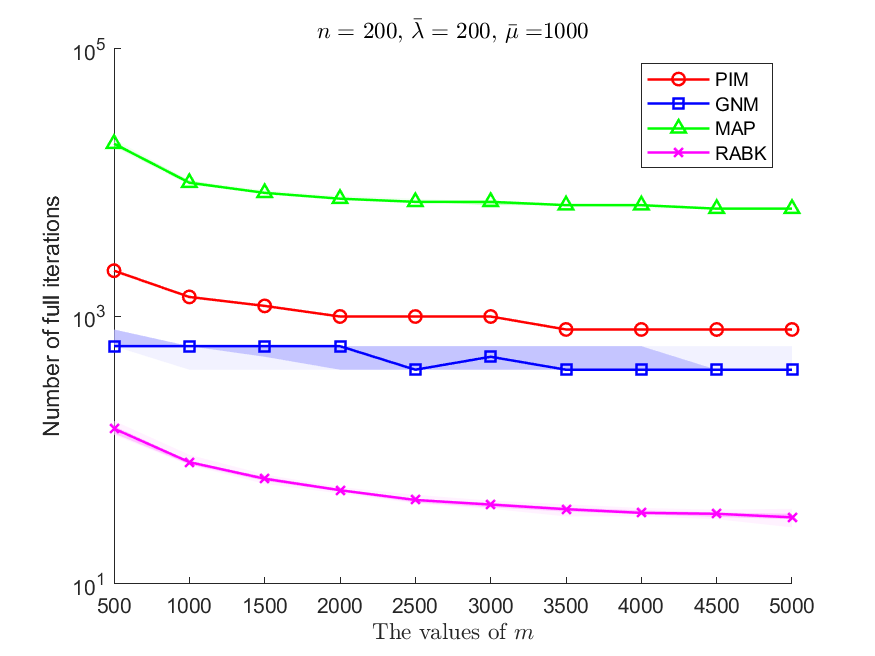}\\
		\includegraphics[width=0.31\linewidth]{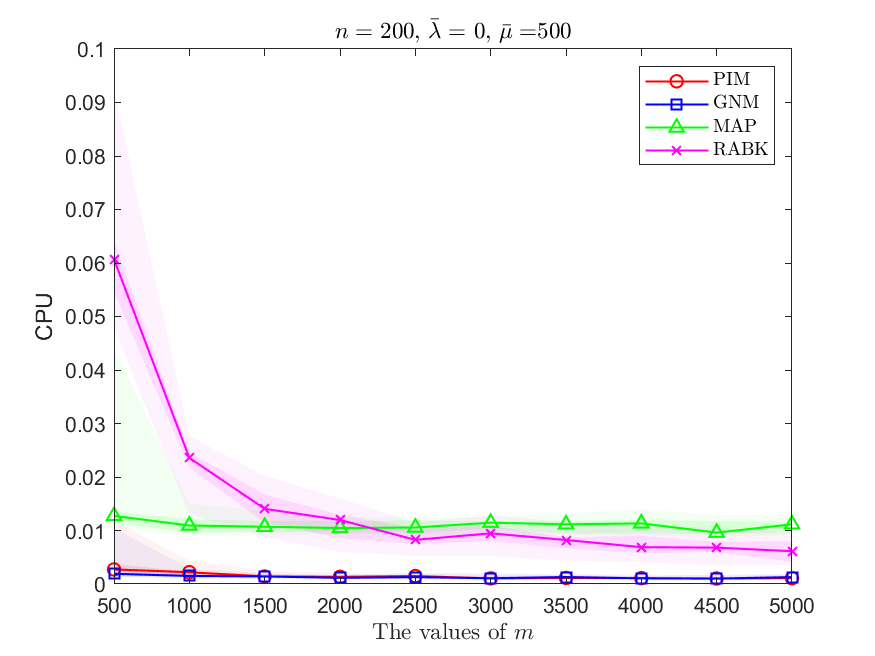}
		\includegraphics[width=0.31\linewidth]{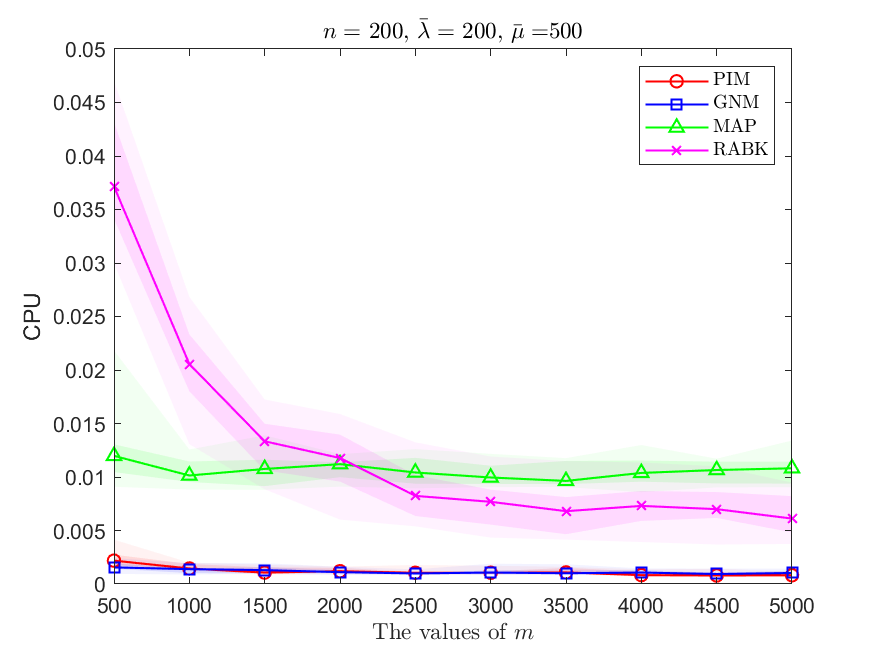}
		\includegraphics[width=0.31\linewidth]{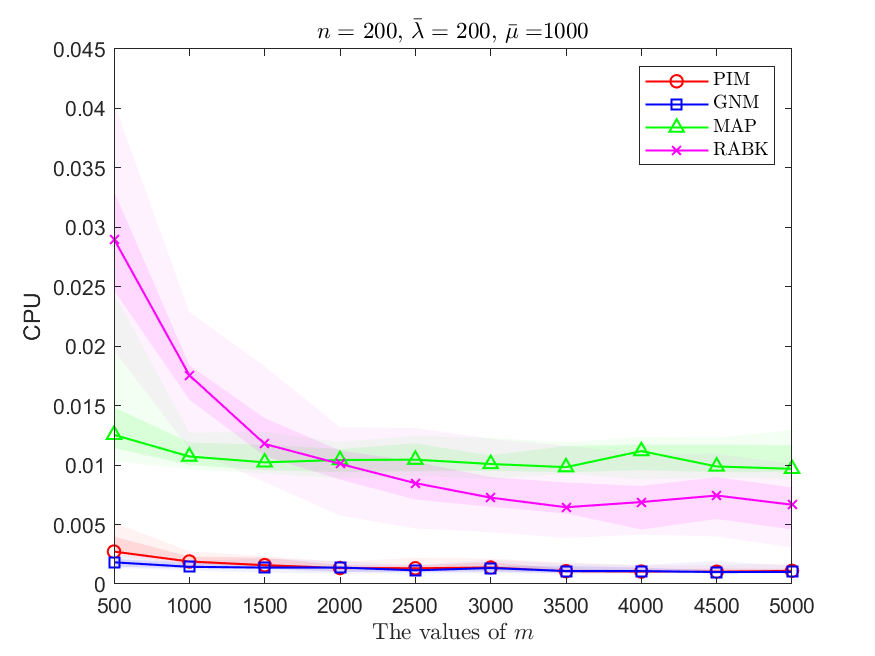}
	\end{tabular}
	\caption{Comparison of PIM, GNM, MAP, and RABK for the asymmetric ridge regression. Figures depict the number of full iterations (left) and CPU time (right) vs increasing of $m$. The title of each plot indicates the values of $n$,  $\bar{\lambda}$, and $\bar{\mu}$. }
	\label{figure4}
\end{figure}

\section{Concluding remarks}
We have proposed a simple and versatile randomized iterative algorithmic framework for solving GAVE, applicable to both square and non-square coefficient matrices. By manipulating the probability spaces, our algorithmic framework can recover a wide range of popular algorithms, including  the PIM and the RK method. The flexibility of our approach also enables us to modify the parameter matrices to create entirely new methods tailored to specific problems. We have provided conditions that ensure the unique solvability and established error bounds for GAVE. Numerical results confirm the efficiency of our proposed method. 

This work not only broadens the applicability of randomized iterative methods, but it also opens up new possibilities for further algorithmic innovations for GAVE.
The momentum acceleration technique, a strategy recognized for its effectiveness in enhancing the performance of optimization methods \cite{Han2022-xh,loizou2020momentum,zeng2023adaptive}, presents a promising area of exploration. Investigating a momentum variant of randomized iterative methods for the GAVE problem could yield significant advancements in this field.
Moreover, real-world applications often present challenges in the form of uncertainty or noise in the input data or computation. Inexact algorithms have been recognized for their ability to provide robustness and flexibility in handling such uncertainties \cite{chen2023exact,loizou2020convergence}. Analyzing inexact variants of randomized iterative methods for solving GAVE would also be a valuable topic.

\bibliographystyle{plain}
\bibliography{references}

\section{Appendix. Proof of the main results}
\label{sec:appd}


%

\subsection{Proof of Theorem \ref{THM-suf}}
The following lemma is essential for proving Theorem \ref{THM-suf}.
\begin{lemma}
	\label{lemma-exist}
	Suppose that $m\leq n$. If there exists a subset $\mathcal{S}\subset [n]$ with $|\mathcal{S}|=m$ such that for any $D\in[-I_m,I_m]$, the matrix $A_{:,\mathcal{S}} + B_{:,\mathcal{S}}D$ is nonsingular, then for any $b\in\mathbb{R}^m$, the GAVE \eqref{GAVE} 
	is solvable. In particular,  the  GAVE \eqref{GAVE} has a unique solution  for any $b\in\mathbb{R}^m$ if and only if $m=n$.
\end{lemma}



\begin{proof}
	We can rewrite \eqref{GAVE} as
	\begin{equation}\label{equ-gave}
		\begin{bmatrix}
			A_{:,\mathcal{S}} \
			A_{:,\mathcal{S}^c}
		\end{bmatrix}
		\begin{bmatrix}
			x_{\mathcal{S}} \\
			x_{\mathcal{S}^c}
		\end{bmatrix}
		-\begin{bmatrix}
			B_{:,\mathcal{S}} \
			B_{:,\mathcal{S}^c}
		\end{bmatrix}\begin{bmatrix}
			|x_{\mathcal{S}}| \\
			|x_{\mathcal{S}^c}|
		\end{bmatrix}=b
	\end{equation}
	It is follows from Theorem \ref{unique-solution} that 
	$$A_{:,\mathcal{S}} x_{\mathcal{S}} -B_{:,\mathcal{S}}|x_{\mathcal{S}} |=b$$
	has a unique solution, denoted as  $x_{\mathcal{S}}^*$. Therefore, $(x_{\mathcal{S}}^*,0)^\top$ is a solution to \eqref{equ-gave}, i.e., a solution to the  GAVE \eqref{GAVE}. Hence, the  GAVE \eqref{GAVE} is solvable. 
	If $m=n$, Theorem \ref{unique-solution} indicates that the  GAVE \eqref{GAVE} now has unique solution. On the other hand, if the  GAVE \eqref{GAVE} has unique solution for any $b\in\mathbb{R}^m$,  Theorem \ref{thm-solvability} ensures that $m=n$. This completes the proof of this lemma.
\end{proof}

Now we are ready to prove Theorem \ref{THM-suf}.

\begin{proof}[Proof Theorem \ref{THM-suf}]
	$(\romannumeral1)$ Suppose that there exists a $\bar{b}\in\mathbb{R}^m$ such the GAVE \eqref{GAVE} is unsolvable. By Lemma \ref{lemma-exist}, we know that for every subset $\mathcal{S}\subset [n]$ with $|\mathcal{S}|=m$, there exists a $D_{\mathcal{S}}\in[-I_m,I_m]$ such that $A_{:,\mathcal{S}} + B_{:,\mathcal{S}}D_{\mathcal{S}}$ is singular.  
	Then we have
	$$
	0=\sigma_{m}\left(MA_{:,\mathcal{S}} + MB_{:,\mathcal{S}}D_{\mathcal{S}}\right)
	\geq\sigma_{m}\left(MA_{:,\mathcal{S}}\right)- \|M B_{:,\mathcal{S}}D_{\mathcal{S}}\|_2
	\geq \sigma_{m}\left(MA\right)-\|MB\|_2,
	$$
	where the last inequality follows from the Cauchy interlacing theorem. This contradicts to the assumption that $\sigma_{m}\left(MA\right)>\|MB\|_2$. Hence, we know that the GAVE \eqref{GAVE} is solvable for each $b\in\mathbb{R}^m$.
	
	$(\romannumeral2)$ Following a similar argument as in \((\romannumeral1)\), we know that the GAVE \eqref{GAVE} is solvable for any \(b \in \mathbb{R}^m\). Together with Lemma \ref{lemma-exist}, this implies that the solution is unique.
	
	$(\romannumeral3)$	Let	$x_1,x_2\in \mathcal{X}^*$, we have
	$$
	\begin{aligned}
		0=&\|MAx_1-MB|x_1|-MAx_2+MB|x_2|\|_2\\
		=&\|MA(x_1-x_2)-MB(|x_1|-|x_2|)\|_2\\
		\geq&\|MA(x_1-x_2)\|_2-\|MB(|x_1|-|x_2|)\|_2\\
		\geq &\left( \sigma_{n}(MA)-\|MB\|_2\right)\|x_1-x_2\|_2\\
		\geq&0.
	\end{aligned}
	$$
	Hence, $x_1=x_2$. This completes the proof of this theorem.
\end{proof}

\subsection{Proof of Theorem \ref{EB-full-rank}}
The following lemmas are useful for our proof.

\begin{lemma}[\cite{hiriart1980mean}, Theorem $8$]
	\label{lemma-MVT}
	Let $F:\mathbb{R}^{n}\to\mathbb{R}^m$ be a locally Lipschitz function on open subset $\Omega$ of $\mathbb{R}^n$ and let  $x,y\in\Omega$. Then
	there exist $\xi_i\geq 0$ with $\sum_{i=1}^m\xi_i=1$, vectors $z_i\in \operatorname{co}(\{x,y\})$, and matrices $M_i \in \partial F(z_i),i=1,\ldots,m$, such that
	$$
	F(x)-F(y)=\sum_{i=1}^{m}\xi_i M_i(x-y).
	$$
\end{lemma}


\begin{lemma}\label{EB-MV}
	For any $x,y\in\mathbb{R}^n$, there exists $\hat{D}:=\hat{D}_{x,y}\in[-I_n,I_n]$  such that 
	$$
	F_{A,B}(x)-F_{A,B}(y)=(A+B\hat{D})(x-y).
	$$
\end{lemma}
\begin{proof}
	The idea of the proof is similar to that of Theorem 7 in \cite{zamani2023error}.
	By the
	mean value theorem, see Lemma \ref{lemma-MVT}, there exist $z_i\in \text{co}(\{x, y\}), M_i\in \partial F_{A,B}(z_i),$ and $ \xi_i \geq 0, i = 1,\ldots,m $ with $\sum_{i=1}^{m}\xi_i=1$ such that
	$$
	F_{A,B}(x)-F_{A,B} (y)=\sum_{i=1}^{m}\xi_i M_i(x-y).
	$$
	Note that for any $z\in \mathbb{R}^n$, $\partial F_{A,B}(z) \subseteq \{A + B\text{diag}(d) \mid \|d\|_{\infty} \leq 1\}$. Combining this result with the convexity of $\{A + B\text{diag}(d) \mid \|d\|_{\infty} \leq 1\}$, we know that there exists $\hat{D}\in[-I_n,I_n]$ such that $\sum_{i=1}^{m}\xi_i M_i=A+B\hat{D}$ and hence
	$$
	F_{A,B}(x)-F_{A,B}(y)=(A+B\hat{D})(x-y).
	$$
	This completes the proof of this lemma.
\end{proof}

\begin{proof}[Proof of Theorem \ref{EB-full-rank}]
	From Lemma \ref{EB-MV}, we know that there exists a $\hat{D}$ such that $M(A+B\hat{D})(x-x^*)=M(Ax-B|x|-b)$. Since $A+BD$ is full column rank, we have
	$$x-x^*=(MA+MB\hat{D})^{\dagger}M(Ax-B|x|-b).$$
	Thus
	$$
	\begin{aligned}
		\| x-x^*\|&=\|(MA+MB\hat{D})^{\dagger}M(Ax-B|x|-b)\|
		\\
		&\leq \max\limits_{D\in[-I_n,I_n]}\|\left(MA+MBD\right)^{\dagger}\|\cdot \|M(Ax-B|x|-b)\|,
	\end{aligned}
	$$
	which completes the proof. 
\end{proof}


\subsection{Proof of Theorem \ref{THM-asym}}

To prove Theorem \ref{THM-asym}, the following supermartingale convergence lemma is required.

\begin{lemma}[\cite{williams1991probability}, Supermartingale convergence lemma]\label{supermartingale}
	Let $v^k$ and $u^k$ be sequences of nonnegative random variables such that $\mathbb{E}[v^{k+1}\mid\mathcal{F}_k]\leq v^k-u^k$ a.s. for all $k\geq0$, where $\mathcal{F}_k$ denotes the collection $\{v^0,\ldots,v^k,u^0,\ldots,u^k\}$. Then, $v^k$ converges to a random variable $v$ a.s. and $\sum_{k=0}^{\infty}u^k<\infty$.
\end{lemma}

\begin{proof}[Proof of Theorem \ref{THM-asym}]
	$(\romannumeral1)$ For any $x^*\in\mathcal{X^*}$, we have
	$$
	\begin{aligned}
		\|x^{k+1}-x^*\|^2_2=&\left\|x^k-x^*-\alpha\frac{A^\top S_kS_k^\top (Ax^k-B|x^k|-b)}{\|S_k^\top A\|^2_2}\right\|^2_2\\
		=&\|x^k-x^*\|^2_2+\alpha^2\frac{\|A^\top S_kS_k^\top (Ax^k-B|x^k|-b)\|_2^2}{\|S_k^\top A\|^4_2}\\
		&-\frac{2\alpha}{\|S_k^\top A\|^2_2}\left\langle S_k^\top (Ax^k-B|x^k|-b),S_k^\top A(x^k-x^*)\right\rangle\\
		\leq &\|x^k-x^*\|^2_2+\alpha^2\frac{\|S_k^\top (Ax^k-B|x^k|-b)\|_2^2}{\|S_k^\top A\|^2_2}\\
		&-\frac{2\alpha}{\|S_k^\top A\|^2_2}\left\langle S_k^\top (Ax^k-B|x^k|-b),S_k^\top A(x^k-x^*)\right\rangle
		\\
		= &\|x^k-x^*\|^2_2-(\alpha-\alpha^2)\frac{\|S_k^\top (Ax^k-B|x^k|-b)\|_2^2}{\|S_k^\top A\|^2_2}\\
		&+\alpha\frac{\| S_k^\top B(|x^k|-|x^*|)\|^2_2}{\|S_k^\top A\|^2_2}-\alpha\frac{\|S_k^\top A(x^k-x^*)\|^2_2}{\|S_k^\top A\|^2_2},
	\end{aligned}
	$$
	where the inequality follows from $\|A^\top S_kS_k^\top (Ax^k-B|x^k|-b)\|_2\leq \|A^\top S_k\|_2\|S_k^\top (Ax^k-B|x^k|-b)\|_2$  and the last equality follows from the facts that $-2\langle a,b\rangle=\|a-b\|^2_2-\|a\|^2_2-\|b\|^2_2$ and $Ax^*-B|x^*|=b$. Taking expectations, we have
	\begin{equation}\label{prove-i-0124}
		\begin{aligned}
			\mathbb{E}[\|x^{k+1}-x^*\|^2_2\mid x^k]\leq &\|x^k-x^*\|^2_2-(\alpha-\alpha^2)\mathbb{E}\left[\frac{\|S^\top (Ax^k-B|x^k|-b)\|_2^2}{\|S^\top A\|^2_2}\right]\\
			&+\alpha\mathbb{E}\left[\frac{\| S^\top B(|x^k|-|x^*|)\|^2_2}{\|S^\top A\|^2_2}\right]-\alpha\mathbb{E}\left[\frac{\|S^\top A(x^k-x^*)\|^2_2}{\|S^\top A\|^2_2}\right]\\
			=&\|x^k-x^*\|^2_2-(\alpha-\alpha^2)\|Ax^k-B|x^k|-b\|_H^2\\
			&+\alpha\| B(|x^k|- |x^*|)\|^2_H-\alpha\| A(x^k- x^*)\|^2_H
			\\
			\leq& \left(1-\alpha\left(\sigma_{n}^2(H^{\frac{1}{2}}A)-\|H^{\frac{1}{2}}B\|^2_2\right)\right)\|x^k-x^*\|^2_2
			\\&-(\alpha-\alpha^2)\|Ax^k-B|x^k|-b\|_H^2,
		\end{aligned}
	\end{equation}
	where the last inequality follows from the fact   $\| B(|x^k|- |x^*|)\|^2_H\leq\|H^{\frac{1}{2}}B\|^2_2\|x^k-x^*\|^2_2$ and $\| A(x^k- x^*)\|^2_H\geq\sigma_{n}^2(H^{\frac{1}{2}}A)\|x^k-x^*\|^2_2$. According to the supermartingale convergence lemma,  $\|x^k-x^*\|_2$ converges a.s. for every $x^*\in\mathcal{X^*}$. Thus, the sequence $\{x^k\}_{k\geq0}$ is bounded a.s., leading to the existence of accumulation points for $\{x^k\}_{k\geq0}$.  Furthermore, as $\sum_{k=1}^{\infty}\|Ax^k-B|x^k|-b\|^2_H<\infty$ a.s., it implies that $\|Ax^k-B|x^k|-b\|_H\to 0$ a.s..
	From Lemma \ref{positive}, we know that $H$ is positive definite and hence $\|Ax^k-B|x^k|-b\|_2\to 0$ a.s., i.e. $Ax^k-B|x^k|-b\to 0$ a.s.. By the continuity of the function $\|Ax-B|x|-b\|_2$, for any accumulation point $\tilde{x}^*$ of $\{x^k\}_{k\geq0}$, we have $A\tilde{x}^*-B|\tilde{x}^*|-b= 0$ a.s.. This implies $\tilde{x}^*\in\mathcal{X^*}$ a.s.. 
	
	$(\romannumeral2)$ From Theorem \ref{THM-suf}, we know that $x^*$ is the unique solution.
	By Corollary \ref{EB-2-norm}, we know that $\|Ax^k-B|x^k|-b\|_H^2\geq \left(\sigma_{n}(H^{\frac{1}{2}}A)-\|H^{\frac{1}{2}}B\|_2\right)^2\|x-x^*\|_2^2$, which together with \eqref{prove-i-0124} and $\alpha=(2-\xi)\frac{\sigma_{n}(H^{\frac{1}{2}}A)}{\sigma_{n}(H^{\frac{1}{2}}A)-\|H^{\frac{1}{2}}B\|_2}\in(0,1]$ implies
	\begin{equation}
		\label{proof-xie-2025-2-3-1}
		\mathbb{E}[\|x^{k+1}-x^*\|^2_2\mid x^k]\leq \left(1-(2-\xi)\xi\sigma_{n}^2(H^{\frac{1}{2}}A)\right)\|x^{k}-x^*\|^2_2.
	\end{equation}
	Taking the expectation over the entire history we have
	$$
	\mathbb{E}[\|x^{k+1}-x^*\|^2_2]\leq\left(1-(2-\xi)\xi\sigma_{n}^2(H^{\frac{1}{2}}A)\right) \mathbb{E}[\|x^{k}-x^*\|^2_2].
	$$
	By induction on the iteration index $k$, we can obtain the desired result.
	
	$(\romannumeral3)$ 
	Let $x_k^{*}\in \operatorname{Proj}_{\mathcal{X}^*}(x^k)$. Since 
	$$
	\operatorname{dist}^2_{\mathcal{X^*}}(x^{k+1})=\|x^{k+1}-x^*_{k+1}\|^2_2\leq\|x^{k+1}-x^*_k\|^2_2,
	$$
	by using the similar arguments as \eqref{prove-i-0124}, we have  
	\begin{equation}\label{dist-re}
		\begin{aligned}
			\mathbb{E}[\operatorname{dist}^2_{\mathcal{X^*}}(x^{k+1})\mid x^k]\leq & \left(1-\alpha\left(\sigma_{n}^2(H^{\frac{1}{2}}A)-\|H^{\frac{1}{2}}B\|^2_2\right)\right) \operatorname{dist}^2_{\mathcal{X^*}}(x^{k})
			\\&-(\alpha-\alpha^2)\|Ax^k-B|x^k|-b\|_H^2.
		\end{aligned}
	\end{equation}
	By the assumptions in this theorem and Theorem \ref{EB-general}, we know that 
	$$
	\|Ax^k-B|x^k|-b\|_H^2\geq \lambda_{\min}(H) \kappa^2\operatorname{dist}^2_{\mathcal{X^*}}(x^{k}),
	$$
	which together with \eqref{dist-re} and $\alpha=(2-\xi)\frac{\sigma_{n}^2(H^{\frac{1}{2}}A)-\|H^{\frac{1}{2}}B\|^2_2+\lambda_{\min}(H)\kappa^2}{2\lambda_{\min}(H)\kappa^2}\in(0,1]$ leads to
	$$
	\mathbb{E}[\operatorname{dist}^2_{\mathcal{X^*}}(x^{k+1})\mid x^k]\leq \left(1-(2-\xi)\xi\frac{\sigma_{n}^2(H^{\frac{1}{2}}A)-\|H^{\frac{1}{2}}B\|^2_2+\lambda_{\min}(H)\kappa^2}{4\lambda_{\min}(H)\kappa^2}\right)
	\operatorname{dist}^2_{\mathcal{X^*}}(x^{k}).
	$$
	By taking expectation over the entire history and  the induction on the iteration index $k$, we can obtain the desired result.
\end{proof}

\section{Appendix. Relationships between the standard AVE and the LCP}
\label{sec:appd2}

This section examines the equivalence between the standard AVE, where \(m=n\) and \(B=I\), and the LCP. It also provides a concise survey on this topic.  The LCP \cite{cottle2009linear} consists of finding $w,z\in\mathbb{R}^\ell$ such that
\begin{equation}
	\label{lcp-f}
	w=Qz+q, \ \ w^\top z=0,  \ \ w,z\geq 0.
\end{equation}
As a fundamental and well-studied problem in mathematical programming, the LCP has strong connections to the AVE, as demonstrated in \cite{prokopyev2009equivalent,mangasarian2007absolute,hu2010note,mangasarian2006absolute}. These problems are not only equivalent but also serve as a basis for deriving essential properties of the AVE.

\subsection{LCP as an AVE}

We introduce the reduction from \cite{mangasarian2007absolute}. Assume that \(Q-I\) is nonsingular; this assumption is general, as  \(Q\) and $q$ in \eqref{lcp-f} can be scaled by any positive scalar. The reduction is based on the substitutions \(w = |x| - x\) and \(z = |x| + x\). Thus, the LCP \eqref{lcp-f} can be expressed as
\[
|x| - x = Q|x| + Qx + q,
\]
which is equivalent to the AVE
\[
(Q-I)^{-1}(I+Q)x - |x| = (Q-I)^{-1}q.
\]

\subsection{AVE as an LCP}

There are primarily three methods to reformulate the AVE into the LCP as described in \cite{prokopyev2009equivalent,hu2010note,mangasarian2006absolute}. These reformulations are based on different mathematical formulations and assumptions, each offering certain advantages in specific contexts.

\textbf{Reformulation I:} This reformulation, proposed by Mangasarian and Meyer in \cite{mangasarian2006absolute}, is based on the assumption that \(A - I\) is nonsingular. 
We denote the positive and negative parts of $x$ by $x^+$ and $x^{-}$, respectively, which gives us
$$
x=x^+-x^{-}, \ |x|=x^++x^-, \ (x^+)^\top x^-=0, \  \ x^+,x^-\geq 0.
$$
Thus, the standard AVE takes the form
$$
A(x^+-x^-)-x^+-x^-=b, \  (x^+)^\top x^-=0, \ \ x^+,x^-\geq 0.
$$
This can then be reformulated into the LCP as 
$$
x^{+}=(A-I)^{-1}(A+I)x^{-}+(A-I)^{-1}b, \  (x^+)^\top x^-=0, \  \ x^+,x^-\geq 0.
$$

\textbf{Reformulation II:}  To eliminate the assumption on the coefficient matrix $A$  in Reformulation I, Prokopyev proposed an alternative approach in \cite{prokopyev2009equivalent}. Consider the LCP \eqref{lcp-f} with \(\ell = 3n\), where \(w = (w_1^\top, w_2^\top, w_3^\top)^\top\) and \(z = (z_1^\top, z_2^\top, z_3^\top)^\top\) with \(w_i, z_i \in \mathbb{R}^n\) for \(i = 1, 2, 3\). The standard AVE can be reformulated into \eqref{lcp-f} with
\[
Q = \begin{bmatrix}
	-I & 2I & 0 \\
	A & -A-I & 0 \\
	-A & A+I & 0
\end{bmatrix}, \quad q = \begin{bmatrix}
	0 \\
	-b \\
	b
\end{bmatrix}.
\]
If \((w^*,z^*)\) solves the related LCP, then \(x^* = z_1^* - z_2^*\) solves the standard AVE.

\textbf{Reformulation III:} It can be seen that Reformulation II increases the problem size, as the number of variables in the related LCP is three times that of the standard AVE. In \cite{hu2010note}, Hu and Huang proposed a new reformulation that neither requires assumptions on \(A\) nor increases the problem size. However, this reformulation relies on a linear independence-based index selection, as described in Algorithm \ref{LCP-index=r}.


\begin{algorithm}[htpb]
	\caption{Linear independence-based index selection}
	\label{LCP-index=r}
	\begin{algorithmic}
		\Require $P, Q \in \mathbb{R}^{m \times m}$ with $\operatorname{rank}(P+Q)=m$, $k=1$, and initialize empty sets $\mathcal{I}=\emptyset$ and $\mathcal{M}=\emptyset$.
		\begin{enumerate}
				\item[1:] If $\mathcal{M}\cup\{P_{:,k},P_{:,k+1}+Q_{:,k+1},\ldots,P_{:,m}+Q_{:,m}\}$ is linearly independent
				
						\qquad	Update $\mathcal{I}=\mathcal{I}\cup\{k\}$ and $\mathcal{M}=\mathcal{M}\cup \{P_{:,k}\}$.
						
						Otherwise 
						
						\qquad	Update $\mathcal{M}=\mathcal{M}\cup \{Q_{:,k}\}$.
			\item[2:] If $k=m$, stop and go to output. Otherwise, set $k=k+1$ and go to Step 1.
		\end{enumerate}
		\Ensure
		The index set  $\mathcal{I}$.
	\end{algorithmic}
\end{algorithm}

For the standard AVE, let the index set \(\mathcal{I}\) be generated by Algorithm \ref{LCP-index=r} with \(P = A + I\) and \(Q = I - A\). Define \(D\) as a diagonal matrix where the \(i\)-th diagonal element is \(1\) if \(i \in \mathcal{I}\) and \(-1\) otherwise. The standard AVE can then be reduced to the following LCP:
\[
w = (AD - I)^{-1}(AD + I)z + 2(AD - I)^{-1}b, \quad w^\top z = 0, \quad w, z \geq 0.
\]
Let \((w^*, z^*)\) be the solution to the reduced LCP. We define \((y^*, s^*)\) as follows:
\[
y^*_i = 
\begin{cases} 
	w^*_i, & \text{if } i \in \mathcal{I}, \\
	z^*_i, & \text{if } i \in \mathcal{I}^c,
\end{cases} 
\quad \text{and} \quad 
s^*_i = 
\begin{cases} 
	z^*_i, & \text{if } i \in \mathcal{I}, \\
	w^*_i, & \text{if } i \in \mathcal{I}^c.
\end{cases}
\]
Then \(x^* = \frac{1}{2}(y^* - s^*)\) is a solution to the standard AVE.

\end{document}